\documentclass{article}

\usepackage{authblk}

\usepackage{hyperref}

\usepackage{palatino}

\usepackage{amsmath}
\usepackage{amsfonts,amssymb,amsmath,latexsym,wasysym,mathrsfs,bbm,stmaryrd}
\usepackage[all]{xy}
\usepackage[usenames]{color}
\usepackage{epsfig}

\usepackage{tikz-cd}
\usetikzlibrary{arrows.meta}
\tikzcdset{arrow style=tikz,diagrams={line width=2pt,>={Stealth[round,length=8pt,width=5pt,inset=2.75pt]}}}

\usepackage{graphicx}
\usepackage{subcaption}

\usepackage{amsthm}
\usepackage{thmtools}
\usepackage{thm-restate}
\declaretheorem[numberwithin=section]{theorem}
\declaretheorem[sibling=theorem]{proposition}
\declaretheorem[sibling=theorem]{definition}
\declaretheorem[sibling=theorem]{corollary}
\declaretheorem[sibling=theorem]{lemma}

\declaretheorem[sibling=theorem]{assumption}

\declaretheorem[sibling=theorem,style=remark]{remark}

\numberwithin{equation}{section}

\usepackage[mathcal]{euscript}
\usepackage{times}

\newcommand{\R}{\mathbb R}
\newcommand{\C}{\mathbb C}
\newcommand{\N}{\mathbb N}
\newcommand{\A}{\mathscr A}

\long\def\symbolfootnote[#1]#2{\begingroup%
\def\thefootnote{\fnsymbol{footnote}}\footnote[#1]{#2}\endgroup}

\begin{document}

	\title{The Brown measure of the sum of a self-adjoint element and an elliptic element}
	\author{Ching-Wei Ho\thanks{\noindent Address: Institute of Mathematics, Academia Sinica, Taipei 10617, Taiwan; Department of Mathematics, University of Notre Dame, Notre Dame,
	IN 46556, United States \\
	\noindent Email: chwho@gate.sinica.edu.tw}}

	\date{\today} 
	
	\maketitle

	\begin{abstract}		
		We completely determine the Brown measure of the sum of a self-adjoint element and an elliptic element, which is the limiting eigenvalue distribution of the random matrix
		\[Y_N+\sqrt{s-\frac{t}{2}}X_N+i\sqrt{\frac{t}{2}}X_N'\]
		where $Y_N$ is an $N\times N$ deterministic Hermitian matrix whose eigenvalue distribution converges as $N\to\infty$ and $X_N$ and $X_N'$ are independent Gaussian unitary ensembles. We also study various asymptotic behaviors of this Brown measure as the variance of the elliptic element approaches infinity. 
		\end{abstract}
	
		\tableofcontents
	
	\section{Introduction}
	\subsection{The sum of a self-adjoint element and an elliptic element}
	An elliptic element is an element in a $W^*$-probability space of the form $z = x+iy$ where $x$ and $y$ are freely independent semicircular elements, possibly with different variances. By substracting the mean $\tau(z)$ if necessary, we only consider the case $\tau(z) = 0$ in this paper. The variance of such an element is given by
	\[\tau(z^*z) = \tau(x^*x)+\tau(y^*y).\]
	Once the variance of $z$ is given, say $s$, there are several possibilities for the variances of $x$ and $y$. We use the parameters $t = 2\tau(y^*y)$, and $\tau(x^*x) = s-\frac{t}{2}$. Under the parameters $s$, $t$, the elliptic element $z$ then has the form
	\[\tilde{\sigma}_{s-\frac{t}{2}}+i\sigma_{\frac{t}{2}}\]
	where $\tilde\sigma_{s-\frac{t}{2}}$ and $\sigma_{\frac{t}{2}}$ are freely independent centered semicircular elements with variances $s-\frac{t}{2}$ and $\frac{t}{2}$ respectively in a certain $W^*$-probability space.
	
	Suppose that $y_0$ is a bounded self-adjoint element in the $W^*$-probability space containing $\tilde\sigma_{s-\frac{t}{2}}$ and $\sigma_{\frac{t}{2}}$; suppose also that all the three elements are freely independent. In this paper, we compute the Brown measure of the element
	\[y_0+\tilde{\sigma}_{s-\frac{t}{2}}+i\sigma_{\frac{t}{2}}.\]
	We show that the Brown measure of $y_0+\tilde{\sigma}_{s-\frac{t}{2}}+i\sigma_{\frac{t}{2}}$ is a push-forward of the Brown measure of $y_0+c_s$ where $c_s = \tilde{\sigma}_{\frac{s}{2}}+i\sigma_{\frac{s}{2}}$ is the Voiculescu's circular element. The Brown measure of $y_0+c_s$ was computed and analyzed by Zhong and the author~\cite{HoZhong2019}. We also study the asymptotic behavior of the Brown measure of $y_0+\tilde{\sigma}_{s-\frac{t}{2}}+i\sigma_{\frac{t}{2}}$ as
	\begin{enumerate}
		\item $s,t\to\infty$ such that the ratio $s/t$ remains as a constant $>\frac{1}{2}$;
		\item $s\to\infty$ and $t$ is kept fixed; and
		\item $s,t\to\infty$ such that the ratio $s/t = \frac{1}{2}$.
		\end{enumerate}

	If $s\geq t$, our results can be computed by the results of Zhong and the author \cite{HoZhong2019} in which the Brown measure of $x_0+c_t$ is computed, with $x_0 = y_0+\tilde{\sigma}_{s-t}$, where $c_t$ is a circular element, freely independent of $x_0$. If $s<t$, $y_0+\tilde{\sigma}_{s-\frac{t}{2}}+i\sigma_{\frac{t}{2}}$ is not a sum of a self-adjoint element and a circular element. We need a more general method.
	
	We use the result introduced in~\cite{HallHo2020} to compute the Brown measure of $y_0+\tilde\sigma_{s-\frac{t}{2}}+i\sigma_{\frac{t}{2}}$ in terms of the Hermitian part $y_0+\tilde\sigma_{s-\frac{t}{2}}$ and $t$ (the parameter of the semicircular element in the skew-Hermitian part). We combine this method with techniques in free probability to determine the Brown measure of $y_0+\tilde{\sigma}_{s-\frac{t}{2}}+i\sigma_{\frac{t}{2}}$ in terms of $y_0$, and $s$ and $t$. The results in~\cite{HallHo2020} used a PDE method introduced in the work of Driver, Hall and Kemp \cite{DHK2019}; this method has been used in subsequent work by other authors~\cite{DemniHamdi2020, HallHo2020, HoZhong2019}. See also the expository article \cite{Hall2019} by Hall for an introduction to the PDE method.

 	Our results have direct connections to random matrix theory. If $X_N$ and $X_N'$ are independent Gaussian unitary emsembles (GUEs), and $Y_N$ is a sequence of $N\times N$ self-adjoint deterministic matrices whose empirical eigenvalue distributions converge weakly to the law of $y_0$, then $Y_N$, $X_N$ and $X_N'$ are asymptotically free in the sense of Voiculescu~\cite{Voiculescu1991}. If $s>\frac{t}{2}$, by~\cite[Theorem 6]{Sniady2002}, the empirical eigenvalue distribution of the (almost surely non-normal) random matrix
	\[Y_N+\sqrt{s-\frac{t}{2}}X_N+i\sqrt{\frac{t}{2}}X_N'\]
	converges to the Brown measure of $y_0+\tilde{\sigma}_{s-\frac{t}{2}}+i\sigma_{\frac{t}{2}}$ as $N\to\infty$. The Brown measure of the case $s=\frac{t}{2}$ is studied in \cite{HallHo2020}, and it is a special case of the results in this paper. In this $s=\frac{t}{2}$ special case, the random matrix model is not a sum of a random matrix and a Ginibre ensemble. We cannot apply \cite{Sniady2002} to conclude that the empirical eigenvalue distribution converges to the Brown measure; it is still an open problem to give a mathematical proof of the convergence. Nevertheless, numerical simulations in \cite{HallHo2020} suggest that the Brown measure of $y_0+i\sigma_{\frac{t}{2}}$ is indeed the limiting eigenvalue distribution of $Y_N+i\sqrt{t/2}X_N$, where $Y_N$ and $X_N$ are the same matrices as above.
		
	The Brown measure computed in the case where $y_0=0$ is the elliptic law \cite{BianeLehner2001} (see also \cite{Girko1985}); its name is due to the fact that its support is a region bounded by an ellipse centered at the origin. In the even more special case $s=t$, the Brown measure is called the circular law since its support is a disk centered at the origin. The circular law was first discovered by Ginibre \cite{Ginibre1965} as a limiting eigenvalue distribution of a random matrix model with Gaussian entries, now commonly called the Ginibre ensemble, then by Girko \cite{Girko1984} in the case when the entries come with more relaxed assumptions. The assumptions of random matrix models were then further relaxed, for example, by Bai \cite{Bai1997}, and Tao and Vu \cite{TaoVu2010}. In the $s\neq t$ case, the elliptic law was first computed by Girko~\cite{Girko1985} as a limiting eigenvalue distribution of a certain random matrix model. The Brown measure, in the operator framework, was computed by Biane and Lehner~\cite{BianeLehner2001} and various later work of others.

	The Brown measure of operators of the form $X+iY$ where $X$ and $Y$ are freely independent has been analyzed at a nonrigorous level in the physics literature. Stephanov \cite{Stephanov1996} used the case when $X$ is Bernoulli distributed and $Y$ is a GUE to provide a model of QCD. Janik et al. \cite{JNPWZ1997} identified the domain where the eigenvalues cluster in the large-$N$ limit when $X$ is an arbitrary self-adjoint random matrix and $Y$ is a GUE. Jarosz and Nowak \cite{JaroszNowak2004, JaroszNowak2006} computed the limiting eigenvalue distribution for general self-adjoint $X$ and $Y$. Belinschi et al.~\cite{BelinschiMaiSpeicher2017, BelinschiSniadySpeicher2018} put the results in~\cite{JaroszNowak2004, JaroszNowak2006} on a more rigorous basis; however, there have not been analytic results about the Brown measure of $X+iY$ obtained under this framework.
	
	Since this article was posted on the arXiv, the results of this article have been extended by several papers. In \cite{Ho2020}, Theorem~\ref{thm:summary1} is extended to the case when $y_0$ is an unbounded self-adjoint element. Zhong \cite{Zhong2021} computes the Brown measure of $y_0+\tilde{\sigma}_{s-\frac{t}{2}}+i\sigma_{\frac{t}{2}}$ for arbitrary bounded operator $y_0$. Hall and the author \cite{HallHo2021} compute the Brown measure of the multiplicative analogue of the operator considered in this paper.

	\subsection{Statements of results}
	Let $y_0$ be a bounded self-adjoint element, $\tilde{\sigma}_{s-\frac{t}{2}}$ and $\sigma_{\frac{t}{2}}$ be semicircular elements with variances $s-t/2$ and $t/2$ in a $W^*$-probability space $(\A, \tau)$, which is a finite von Neumann algebra $\A$ with a faithful, normal, tracial state $\tau$. Suppose also that all three of them are freely independent. Throughout the paper, we let $\nu$ be the law (or distribution) of $y_0$, which is the unique compactly supported probability measure on $\R$ such that
	\[\int x^n\,d\nu(x) = \tau(y_0^n), \quad\textrm{for all }n\in\N.\]
	Recall that, in this paper, we compute the Brown measure of the element
	\[y_0+\tilde{\sigma}_{s-\frac{t}{2}}+i\sigma_{\frac{t}{2}}\in\A.\]
	Background information of free probability and Brown measure is reviewed in Section \ref{sect:Bkground}. The choice of the parameters $s, t$ comes from the context of the two-parameter Segal--Bargmann transform \cite{DriverHall1999, Hall1999, Ho2016}. It is a interpolation between the self-adjoint element $y_0+\sigma_s$ and the element $y_0+i\sigma_s$ studied in \cite{HallHo2020}.

	We make the following standing assumption about the element $y_0+\tilde{\sigma}_{s-\frac{t}{2}}+i\sigma_{\frac{t}{2}}$. We use $\mathrm{Law}(a)$ to denote the law of any self-adjoint random variable $a\in\A$ and $\mathrm{Brown}(a)$ to denote the Brown measure of any non-self-adjoint random variable $a\in\A$.
	\begin{assumption}
		\label{assump:Standing}
		Throughout the paper, we assume either $s>\frac{t}{2}$ or $\nu$ is not a Dirac measure, so that $\mathrm{Law}(y_{0}+\tilde{\sigma}_{s-\frac{t}{2}})$ is not a Dirac measure. 
		\end{assumption}
	When this assumption does not hold, that is, if $\mathrm{Law}(y_{0}+\tilde{\sigma}_{s-\frac{t}{2}})$ is a Dirac measure, then one cannot apply the results from \cite{HallHo2020}. However, in this case, the element $y_0+\tilde{\sigma}_{s-\frac{t}{2}}+i\sigma_{\frac{t}{2}}$ has the form $u\mathbf{1}+i\sigma_{\frac{t}{2}}$ for some constant $u\in\R$ (where $\mathbf{1}$ is the identity element in $\A$). The Brown measure is then a semicircular distribution centered at $u$ with variance $t/2$ on the vertical line through the point $u$. Under Assumption~\ref{assump:Standing}, by the results in \cite{HallHo2020}, the Brown measure is absolutely continuous with respect to the Lebesgue measure on the plane.
	
	The following theorem summarizes Theorems~\ref{thm:pushforward} and~\ref{thm:addellipse}; the proofs can be found in Sections~\ref{sect:pushforward} and~\ref{sect:density}. The results in \cite{HoZhong2019} and \cite{HallHo2020} show that both $\mathrm{Brown}(y_0+\tilde{\sigma}_{s-\frac{t}{2}}+i\sigma_{\frac{t}{2}})$ and $\mathrm{Brown}(y_0+c_s)$ can be pushed forward to $\mathrm{Law}(y_0+\sigma_s)$. Points 2 and 3 of the following theorem are proved by comparing these two push-forward maps. We then use the push-forward result to compute the density of $\mathrm{Brown}(y_0+\tilde{\sigma}_{s-\frac{t}{2}}+i\sigma_{\frac{t}{2}})$ given in Point 1 of the following theorem.
	\begin{theorem}
		\label{thm:summary1}
		\begin{enumerate}
		\item For each $s\geq\frac{t}{2}>0$, there is a continuous function $b_{s,t}:\R\to[0,\infty)$ such that the Brown measure of $y_0+\tilde{\sigma}_{s-t/2}+i\sigma_{t/2}$ is supported in the closure of the set
		 \[\Omega_{s,t} = \{a+ib\in\C|\left\vert b\right\vert<b_{s,t}(a)\}.\]
		 The boundary of $\Omega_{s,t}$ is of measure zero with respect to the Brown measure. The Brown measure is absolutely continuous with respect to the Lebesgue area measure on $\C$, with density
		 \[w_{y_0,s,t}(a+ib) = \frac{1}{2\pi t}\left(1+t\frac{d}{da}\int_\R\frac{(\alpha_{s,t}(a) - x)\,d\nu(x)}{(\alpha_{s,t}(a)-x)^2+v_{y_0,s}(\alpha_{s,t}(a))^2}\right),\] 
		 for $\left\vert b\right\vert<b_{s,t}(a)$, where $\alpha_{s,t}$ is a certain homeomorphism on $\R$ and $v_{y_0,s}$ is a certain nonnegative continuous function on $\R$ such that $\alpha_{s,t}$ and $v_{y_0,s}\circ \alpha_{s,t}$ are differentiable in $\Omega_{s,t}\cap\R$. In particular, the density is constant in the vertical direction.
		 
		 \item The Brown measure of $y_0+\tilde{\sigma}_{s-t/2}+i\sigma_{t/2}$ is the push-forward measure of the Brown measure of $y_0+c_s$ by the homeomorphism $U_{s,t}:\C\to\C$,
		 \[U_{s,t}(\alpha+i\beta)=a_{s,t}(\alpha)+i\frac{t}{s}\beta\]
		 where $a_{s,t}$ is the inverse function of $\alpha_{s,t}$.
		 
		 \item The push-forward measure of the Brown measure of $y_0+\tilde{\sigma}_{s-t/2}+i\sigma_{t/2}$ by the map, constant in the vertical directions,
		 \[Q_{s,t}(a+ib):= \frac{1}{s-t}[sa-t\alpha_{s,t}(a)]\]
		 is the law of the self-adjoint element $y_0+\sigma_s$.
		 \end{enumerate}
	\end{theorem}

	We now describe briefly how to compute the functions $\alpha_{s,t}$, $b_{s,t}$, and $v_{y_0,s}\circ \alpha_{s,t}$ from the above theorem in $\Omega_{s,t}\cap\R$. Given $a\in\R$, we try to solve for $\alpha\in\R$ and $v>0$ the equations
	\begin{equation}
	\label{eq:IntroEq}
	\begin{split}
	&\int\frac{d\nu(x)}{(\alpha-x)^2+v^2} = \frac{1}{s}\\
	&\frac{(2s-t)\alpha}{s}-(s-t)\int\frac{x\,d\nu(x)}{(\alpha-x)^2+v^2} = a.
	\end{split}
	\end{equation}
	The following proposition shows that $a\in\Omega_{s,t}\cap\R$ is precisely when \eqref{eq:IntroEq} has a unique pair of solution. It also shows how the functions $\alpha_{s,t}$, $v_{y_0,s}\circ \alpha_{s,t}$ and $b_{s,t}$ in Theorem~\ref{thm:summary1} are computed using the solution. This proposition is proved in Corollary~\ref{cor:IntroEq}.
	\begin{proposition}
		\label{prop:IntroEq}
		Given any $a\in\R$, \eqref{eq:IntroEq} has a pair of solution $\alpha\in\R$ and $v>0$ if and only if $a\in\Omega_{s,t}\cap\R$. In this case, the solution is unique, and $\alpha_{s,t}(a)=\alpha$, $v_{y_0,s}(\alpha_{s,t}(a))=v$ and $b_{s,t}(a) = \frac{t}{s}v$.
	\end{proposition}
	
	In the special case $s=t$, we obtain $\alpha_{s,t}(a) = a$ and, by Theorem~\ref{thm:summary1},
	\[w_{y_0,s,s}(a+ib) = \frac{1}{\pi s}\left(1-\frac{t}{2}\frac{d}{da}\int_\R\frac{x\,d\nu(x)}{(a-x)^2+v_{y_0,s}(a)^2}\right)\]
	which reduces to the results in \cite{HoZhong2019}. In another special case $t = 2s$, the equations in~\eqref{eq:IntroEq} reduces to (1.4) and (1.5) in \cite{HallHo2020}; the function $\alpha_{s,t}$ is the function $a_0^s$ in \cite{HallHo2020} and the density is given by
	\[\frac{1}{2\pi s}\left(\frac{da_0^s}{da}-\frac{1}{2}\right).\]
	Thus, in the case, Theorem~\ref{thm:summary1} reduces to the results in \cite{HallHo2020}. 

In Sections~\ref{sect:circularAsymp} and~\ref{sect:ellipticAsymp}, we also investigate the asymptotic behaviors of the Brown measure of $y_0+\tilde{\sigma}_{s-\frac{t}{2}}+i\sigma_{\frac{t}{2}}$, which are summarized in the following theorem; roughly speaking, the Brown measure of $y_0+\tilde{\sigma}_{s-\frac{t}{2}}+i\sigma_{\frac{t}{2}}$ behaves like the Brown measure of $\tilde{\sigma}_{s-\frac{t}{2}}+i\sigma_{\frac{t}{2}}$. Point 1 of the following theorem is proved in Theorems~\ref{thm:EllipseDomain} and \ref{thm:EllipticDensity}; Point 2 is proved in Theorem~\ref{thm:thm:FixtAsympDomain} and \ref{thm:FixtAsympDensity}; and Point 3 is proved in Theorem~\ref{thm:Skewasymptotic}. See these theorems for the precise statements. 
	\begin{theorem}
		\label{thm:summary2}
	In all of the following three limiting regimes, the function $b_{s,t}$ is unimodal for all large enough $s$. 
		\begin{enumerate}
			\item As $s,t\to\infty$ such that the ratio $s/t$ remains as a constant $>\frac{1}{2}$: the domain $\Omega_{s,t}$ is asymptotically equivalent to a region bounded an ellipse centered at $(\tau(y_0), 0)$ with horizontal semi-axis of length $\frac{2s-t}{\sqrt{s}}$ and vertical semi-axis of length $\frac{t}{\sqrt{s}}$. The density $w_{y_0,s,t}$ converges to the constant
			\[\frac{1}{\pi}\frac{s}{(2s-t)t}.\]
			 Both convergences are uniform outside any neighborhood of the endpoints of $\Omega_{s,t}\cap\R$.
			\item As $s\to\infty$ and $t$ is kept fixed: the domain $\Omega_{s,t}$ is asymptotically equivalent to a region bounded by a long and thin ellipse centered at $(\tau(y_0),0)$, with horizontal semi-axis of length $2\sqrt{s}$ and vertical semi-axis of length $\frac{t}{\sqrt{s}}$. The density converges to the constant
			\[\frac{1}{2\pi t}.\]
			Both convergences are uniform outside any neighborhood of the endpoints of $\Omega_{s,t}\cap\R$.
			\item As $s,t\to\infty$ such that the ratio $s/t = \frac{1}{2}$: the domain $\Omega_{s,t}$ is asymptotically equivalent to a region bounded a narrow and tall ellipse centered at $(\tau(y_0), 0)$, with vertical semi-axis of length $2\sqrt{s}$. The set $\Omega_{s,t}\cap\R$ concentrates around $\tau(y_0)$; more precisely, given any $c>1$, we have
			\[-\frac{4c\tau(y_0^2)}{\sqrt{s}}<\inf(\Omega_{s,t}\cap\R)-\tau(y_0)<0<\sup(\Omega_{s,t}\cap\R)-\tau(y_0)<\frac{4c\tau(y_0^2)}{\sqrt{s}}.\]
			for all large enough $s$. 
			\end{enumerate}
		\end{theorem}
	We do not have a density estimate for the last case.
	
	\section{Background and previous results\label{sect:Bkground}}
	\subsection{Free random variables}
	\begin{definition}
		\begin{enumerate}
		\item We call $(\A, \tau)$ a {\bf $W^*$-probability space} if $\A$ is a von Neumann algebra and $\tau$ is a normal, faithful tracial state on $\A$. The elements in $\A$ are called {\bf non-commutative random variables}, or simply random variables.
		\item The $\ast$-subalgebras $A_1,\ldots A_n\subset \A$ are said to be {\bf freely independent} if given an $i_1, i_2,\ldots i_m\in\{1,\ldots,n\}$ with $i_k\neq i_{k+1}$, $a_{i_j}\in\A_{i_j}$ are centered, then we also have $\tau(a_{i_1}a_{i_2}\ldots a_{i_m})=0$. The random variables $a_1,\ldots,a_m$ are freely independent if the $\ast$-algebras they generate are freely independent.
		\item For a self-adjoint element $a\in\A$, the {\bf distribution}, or the {\bf law}, of $a$ is a compactly supported measure $\mu$ on $\R$ such that
		\[\int_\R f\,d\mu = \tau(f(a))\]
		for all continuous function $f$. We denote by $\mathrm{Law}(a)$ the law of $a$.
		\end{enumerate}
	\end{definition}

	We now introduce the random variables that are key to this paper. The {\bf semicircular element} $\sigma_t$ has the {\bf semicircular distribution}, or the {\bf semicircle law} of variance $t$, supported on $[-2\sqrt{t},2\sqrt{t}]$ with density
	\[\frac{\sqrt{4t-x^2}}{2\pi t}\,dx.\]
	The {\bf circular element} $c_s$ has the form $\tilde{\sigma}_{\frac{s}{2}}+i\sigma_{\frac{s}{2}}$ where $\tilde{\sigma}_{\frac{s}{2}}$ and $\sigma_{\frac{s}{2}}$ are freely independent semicircular elements. The {\bf elliptic element} has the form $\tilde{\sigma}_{s-\frac{t}{2}}+i\sigma_{\frac{t}{2}}$ where $\tilde{\sigma}_{s-\frac{t}{2}}$ and $\sigma_{\frac{t}{2}}$ are freely independent semicircular elements.

	\subsubsection{The $R$-transform\label{sect:Rtransform}}
	Let $a\in\A$ be a self-adjoint element with law $\mu$. Then we consider the {\bf Cauchy transform}
	\[G_a(z) = \int\frac{1}{z-x}\,d\mu(x)\]
	defined outside the spectrum of $a$. The Cauchy transform $G_a$ is univalent around $\infty$. Denote by $K_a$ the inverse of $G_a$ at $\infty$, and let
	\[R_a(z) = K_a(z)-\frac{1}{z}.\]
	We call $K_a$ the {\bf $K$-transform} of $a$ and $R_a$ the {\bf $R$-transform} of $a$. 

	\begin{theorem}[\cite{Voiculescu1986}]
		If $a_1, a_2\in\A$ are freely independent self-adjoint random variables, then the $R$-transform of the random variable $a=a_1+a_2$ is given by
		\[R_a = R_{a_1}+R_{a_2}.\]
		\end{theorem}
	Using the notations in the theorem, the distribution of $a$ is called the {\bf free convolution} of $a_1+a_2$.

	\subsection{The Brown measure}
	In this section, we review the definition of the Brown measure, which was introduced by Brown \cite{Brown1986}. Let $a\in\A$. We define a function $S$ by
	\[S(\lambda, \varepsilon) = \tau[\log(|a-\lambda|^2+\varepsilon)],\quad \lambda\in\C, \varepsilon>0.\]
	Then 
	\[S(\lambda, 0) = \lim_{\varepsilon\to 0^+}S(\lambda, \varepsilon)\]
	exists as a subharmonic function on $\C$, with value in $\R\cup \{-\infty\}$. The {\bf Brown measure} of $a$, denoted by $\mathrm{Brown}(a)$, is defined to be 
	\[\mathrm{Brown}(a) = \frac{1}{4\pi}\Delta_\lambda S(\lambda, 0)\]
	where the Laplacian is in distributional sense.
	
	One can see that $S(\lambda, 0)$ does define a harmonic function \emph{outside} the spectrum of $a$; the Brown measure of $a$ is a probability measure supported on the spectrum of $a$. The support of $\mathrm{Brown}(a)$, however, can be a proper subset of the spectrum of $a$.
	
	The Brown measure of an $N\times N$ matrix is the empirical eigenvalue distribution of the matrix. If a sequence of random matrices $A_N$ converges in $\ast$-distribution to an element $a$ in a non-commutative probability space, one generally expects that the empirical eigenvalue distribution of $A_N$ converges to the Brown measure of $a$; this, however, is not always the case. A counter-example is the nilpotent matrix
	\[
	\begin{pmatrix}
	0&1&0&\cdots&0\\
	0&0&1&\cdots&0\\
	\vdots&\vdots&\vdots&\ddots&\vdots\\
	0&0&0&\cdots&1\\
	0&0&0&\cdots&0\\
	\end{pmatrix};\]
	this sequence of matrices converges to the Haar unitary element in $\ast$-distribution but the empirical eigenvalue distribution is always the Dirac measure at $0$.
	
	The Brown measure of the circular element $c_s = \tilde{\sigma}_{\frac{s}{2}}+i\sigma_{\frac{s}{2}}$ is called the {\bf circular law} and is supported in the disk of radius $\sqrt{s}$ centered at the origin. The density is the constant
	\[\frac{1}{\pi s}\]
	in the support. The circular element is an $R$-diagonal element. The Brown measure of the circular element can be computed by the method developed by Haagerup and Larsen \cite{HaagerupLarsen2000} and Haagerup and Schultz \cite{HaagerupSchultz2007}. 
	
	The Brown measure of the elliptic element $\tilde{\sigma}_{s-\frac{t}{2}}+i\sigma_{\frac{t}{2}}$ is called the {\bf elliptic law} and is supported in an ellipse with semi-axes on the real and imaginary axes of length $\frac{2s-t}{\sqrt{s}}$ and $\frac{t}{\sqrt{s}}$ respectively. The density is the constant
	\[\frac{1}{\pi}\frac{s}{2s-t}\]
	in the support.
	The elliptic law was computed by Biane and Lehner~\cite{BianeLehner2001}.

	\subsection{Biane's free convolution formula}
	\label{sect:Biane}
	In this section, we review the results of the distribution of the free convolution of a self-adjoint element and a semicircular element established by Biane \cite{Biane1997sc}; several functions and a domain also come up in our study of Brown measure. Given a self-adjoint random variable $x_0$ with law $\mu$, we consider the function
	\[v_{x_0,t}(u) = \inf\left\{v>0\left|\int_\R \frac{d\mu(x)}{(x-u)^2+v^2}> \frac{1}{t}\right.\right\}.\]
	That is, if
	\[\int_\R\frac{d\mu(x)}{(u-x)^2}>\frac{1}{t},\]
	then $v_{x_0,t}(u)$ is defined to be the unique positive number such that
	\begin{equation}
		\label{eq:IntOfSq}
		\int_\R\frac{d\mu(x)}{(u-x)^2+v_{x_0,t}(u)^2} = \frac{1}{t};
		\end{equation}
	otherwise, if
	\[\int_\R\frac{d\mu(x)}{(u-x)^2}\leq\frac{1}{t},\]
	then we set $v_{x_0,t}(u) = 0$. It is noted in \cite{Biane1997sc} that the function $v_{x_0,t}$ is continuous on $\R$ and is differentiable at the points $u$ where $v_{x_0,t}(u)>0$.
	\begin{definition}
		\label{def:BianeFunction}
		We introduce the following notations.
		\begin{enumerate}
			\item $\Delta_{x_0, t} = \{u+iv\in\C | v>v_{x_0,t}(u)\}$ is the region above the graph of $v_{x_0,t}$ in the upper half plane.
			\item $H_{x_0, t}(z) = z+ t G_{x_0}(z)$, $z\in \Delta_{x_0,t}$.
			\end{enumerate}
		\end{definition}
	\begin{theorem}[\cite{Biane1997sc}]
		\label{thm:BianeFC}
		\begin{enumerate}
		\item  The function $H_{x_0, t}$ is an injective conformal map, from $\Delta_{x_0, t}$ \emph{onto} the upper half plane $\C^+$; the function $H_{x_0, t}$ extends to a homeomorphism from the closure $\overline{\Delta_{x_0, t}}$ of $\Delta_{x_0,t}$ onto $\C^+\cup\R$. In particular, $H_{x_0,t}(u+iv_{x_0,t}(u))$ is real.
		\item The function $H_{x_0, t}$ satisfies
		\[G_{x_0+\sigma_t}(H_{x_0, t}(z)) = G_{x_0}(z).\]
		\item The measure $\mathrm{Law}(x_0+\sigma_t)$ is absolutely continuous with respect to the Lebesgue measure; its density $p_{x_0,t}$ can be computed by the function $\psi_{x_0,t}(u) := H_{x_0,t}(u+iv_{x_0,t}(u))$. The function $\psi_{x_0,t}:\R\to\R$ is a homeomorphism, and
		\[p_{x_0,t}(\psi_{x_0,t}(u))=\frac{v_{x_0,t}(u)}{\pi t}.\]
		\item As a consequence, the support of $\mathrm{Law}(x_0+\sigma_t)$ is the closure of the open set $\{\psi_{x_0,t}(u)| v_{x_0,t}(u)>0\}$.
		\end{enumerate}
		\end{theorem}

		\begin{remark}
			\label{rem:Hextension}

			Let $\Lambda_{x_0,t} = \{u+iv\in\C|\left\vert v\right\vert<v_{x_0,t}(u)\}$. The map $H_{x_0, t}$ can be extended to an injective conformal map on $(\overline{\Lambda_{x_0, t}})^c$ by Schwarz reflection with a continuous extension to $\Lambda_{x_0,t}^c$. From now on, $H_{x_0,t}$ means the extension defined on $\Lambda_{x_0,t}^c$. If $v_{x_0, t}(u)>0$, $H_{x_0, t}$ maps both boundary points $u\pm iv_{x_0,t}(u)$ of $\Lambda_{x_0, t}$ to the same point in the support of $\mathrm{Law}(x_0+\sigma_t)$. 
			
			We then define the right inverse $H_{x_0,t}^{-1}$ of $H_{x_0,t}$ as follows. Outside the interior of the support of $\mathrm{Law}(x_0+\sigma_t)$, which is the closure of an open set by Theorem~\ref{thm:BianeFC}(4), $H_{x_0,t}^{-1}$ is defined to be the inverse of $H_{x_0,t}$. Given any $q$ in the interior of the support of $\mathrm{Law}(x_0+\sigma_t)$, we define
		\[H_{x_0, t}^{-1}(q) = u+iv_{x_0,t}(u)\]
		where $u$ is chosen such that $H_{x_0,t}(u+iv_{x_0,t}(u)) = q$. Thus, the restriction of $H_{x_0,t}^{-1}(q)$ to $\C^+\cup\R$ is the inverse of $H_{x_0,t}$ on $\overline{\Delta_{x_0,t}}$.
		\end{remark}

	\subsection{Sum of a self-adjoint and a circular elements}
	In \cite{HoZhong2019}, the author and Zhong computed the Brown measure of $x_0+c_t$, where $x_0$ is a self-adjoint element freely independent of the circular element $c_t$, using the method introduced by Driver, Hall and Kemp \cite{DHK2019}. Interestingly, the support of the Brown measure is bounded by the graph of Biane's function $v_{x_0,t}$ introduced in Section \ref{sect:Biane} and the density is closely related to the law of the self-adjoint element $x_0+\sigma_t$. In this section, we review the results established in \cite{HoZhong2019}.
	
	\begin{theorem}
		\label{thm:HZ}
	Let 
	\begin{equation}
		\label{eq:LambdaDef}\Lambda_{x_0,t} = \{u+iv\in\C|\left\vert v\right\vert<v_{x_0,t}(u)\}.
		\end{equation}
	Then $\Lambda_{x_0,t} $ is a set of full measure with respect to $\mathrm{Brown}(x_0+c_t)$, and its density $w_{x_0,t}$ has the form
	\[w_{x_0,t}(u+iv) = \frac{1}{2\pi t}\frac{d\psi_{x_0,t}(u)}{du}, \quad u+iv\in\Lambda_{x_0,t}\]
	where $\psi_{x_0,t}$ is defined in Theorem \ref{thm:BianeFC}. The density is constant along the vertical segments. 
	
	Furthermore, the push-forward of $\mathrm{Brown}(x_0+c_t)$ by
	\[\Psi_{x_0, t}(u+iv) = H_{x_0,t}(u+iv_{x_0,t}(u)), \quad u+iv\in\Lambda_{x_0, t}\]
	which is independent of $v$, is the law of $x_0+\sigma_t$.
	\end{theorem}

	\subsection{Sum of a self-adjoint and an imaginary multiple of semicircular elements}
	\label{sect:HH}
	Hall and the author computed in \cite{HallHo2020} the Brown measure of $x_0+i\sigma_t$, a sum of a self-adjoint element and an imaginary multiple of semicircular element. The computation of the Brown measure of elements of the form $x_0+i\sigma_t$ covers the case $x_0+c_t$ which has the same $\ast$-moments as $x_0+\sigma_{t/2}+i\tilde{\sigma}_{t/2}$ where $\sigma_{\frac{t}{2}}$ and $\tilde{\sigma}_{\frac{t}{2}}$ are freely independent semicircular elements, both freely independent of $x_0$. The results in \cite{HallHo2020} show that there is a connection between the Brown measure of $x_0+i\sigma_t$, that of $x_0+c_t$ as well as the law of $x_0+\sigma_t$, for the \emph{same} self-adjoint element $x_0$.
	
	We need the following notations to describe the results in \cite{HallHo2020}. 
	\begin{definition} 
		\label{def:H}
		Let $x_0$ be a self-adjoint element.
		\begin{enumerate}
			\item Given any $r\in\R$, let $H_{x_0, r}(z) = z+ r G_{x_0}(z)$, $z\in \Delta_{x_0,|r|}$. Compared to the holomorphic function $H$ in Definition \ref{def:BianeFunction}, we allow $r$ negative in this notation. By the results in \cite{HallHo2020}, for $t>0$, the map $H_{x_0, -t}(z)$ is an injective conformal map on $\Delta_{x_0,t}$ (see Definition \ref{def:BianeFunction} using $x_0$ and the positive $t$, not $-t$). In \cite{HallHo2020}, the authors use the notation $J_t$ instead of $H_{x_0, -t}$. Furthermore, $H_{x_0,r}$ can be extended on $\Lambda_{x_0,s}^c$ by Schwarz reflection.
			\item Define $h_{x_0,t}(u) = \mathrm{Re}[H_{x_0, -t}(u+iv_{x_0,t}(u))]$ on $\R$. This function $h_{x_0,t}$ is a homeomorphism from $\R$ to $\R$; it is a strictly increasing function. If $v_{x_0,t}(u)>0$, we have $h_{x_0,t}'(u)>0$.
			\item Denote by $h_{x_0, t}^{-1}$ the inverse of $h_{x_0, t}$.
		\end{enumerate}
	\end{definition}
	The following theorem established in \cite{HallHo2020} computes the Brown measure of $x_0+i\sigma_t$.
	\begin{theorem}
		\label{thm:HHBrown}
		Let
	\[\Omega_{x_0,t} = [H_{x_0, -t}(\Lambda_{x_0,t}^c)]^c.\]
	 Then we can write $\Omega_{x_0,t}$ as
	\[\Omega_{x_0, t} = \{a+ib\in\C|\left\vert b\right\vert < b_{x_0,t}(a)\}\]
	where $b_{x_0,t}(a) = 2v_{x_0,t}(h_{x_0,t}^{-1}(a))$ is a nonnegative function on $\R$. The set $\Omega_{x_0,t}$ itself is a set of full measure with respect to $\mathrm{Brown}(x_0+i\sigma_t)$. 
	
	Inside $\Omega_{x_0,t}$, $\mathrm{Brown}(x_0+i\sigma_t)$ is absolutely continuous with respect to the Lebesgue measure on the plane with a strictly positive density; the density has the form
	\[\frac{1}{2\pi t}\left(\frac{dh_{x_0,t}^{-1}(a)}{da}-\frac{1}{2}\right), \quad a+ib\in\Omega_{x_0,t}.\]
	In particular, the density is independent of $b$ and is constant along the vertical segments.
	\end{theorem}
	
	We now describe the connections of $\mathrm{Brown}(x_0+c_t)$, $\mathrm{Brown}(x_0+i\sigma_t)$, and $\mathrm{Law}(x_0+\sigma_t)$. Let $U_{x_0, t}:\overline{\Lambda_{x_0,t}}\to\overline{\Omega_{x_0,t}}$ be a homeomorphism defined by
	\[U_{x_0, t}(u+iv) = h_{x_0,t}(u)+2iv.\]
	Note that the map $U_{x_0,t}$ takes the vertical line segments in $\overline{\Lambda_{x_0,t}}$ \emph{linearly} to vertical line segments in $\overline{\Omega_{x_0, t}}$. Also, recall that $\Lambda_{x_0, t}$ defined in \eqref{eq:LambdaDef} is an open set of full measure of $\mathrm{Brown}(x_0+c_t)$.  The following theorem establishes the push-forward relations between $\mathrm{Brown}(x_0+c_t)$, $\mathrm{Brown}(x_0+i\sigma_t)$ and $\mathrm{Law}(x_0+\sigma_t)$. It is proved in \cite{HallHo2020}.
	\begin{theorem}
		\label{thm:HHpushforward}
	\begin{enumerate}
		\item The push-forward measure of $\mathrm{Brown}(x_0+c_t)$ under $U_{x_0,t}$ is the Brown measure $\mathrm{Brown}(x_0+i\sigma_t)$.
	
	\item The push-forward of $\mathrm{Brown}(x_0+i\sigma_t)$ under the map
	\begin{equation}
		\label{eq:Qtmap}
		Q_{x_0, t}(a+ib) := 2h_{x_0,t}^{-1}(a)-a 
		\end{equation}
	is the law of $x_0+\sigma_t$. 	The map $Q_{x_0,t}$ agrees with $\Psi_{x_0, t}\circ U_{x_0,t}^{-1}$ where $\Psi_{x_0, t}$ is defined in Theorem~\ref{thm:HZ}. Alternatively, by Definition 8.1 of \cite{HallHo2020}, we can write
	\[Q_{x_0,t}(a+ib) = H_{x_0,t}\circ H_{x_0,-t}^{-1}(a+ib_{x_0,t}(a)), \quad a\in \Omega_{x_0,t}.\]
	Moreover, $Q_{x_0,t}$ is a diffeomorphism on $\Omega_{x_0,t}\cap\R$.
	\end{enumerate}
	\end{theorem}

	Although $Q_{x_0,t}$ is not an invertible map, Point 2 of Theorem~\ref{thm:HHpushforward} characterizes the probability measure on $\Omega_{x_0,t}$ whose density is constant along vertical segments. Similar results of the following proposition for the Brown measures of different random variables can be found in \cite{DHK2019,HoZhong2019}.
	\begin{proposition}
		\label{prop:HHUnique}
		The Brown measure of $x_0+i\sigma_t$ is the unique measure $m$ on $\overline{\Omega_{x_0,t}}$ that is absolutely continuous with respect to the Lebesgue measure such that the density is constant along vertical segments and the push-forward of $m$ by $Q_{x_0,t}$ is $\mathrm{Law}(x_0+\sigma_t)$.
	\end{proposition}
	\begin{proof}
		Suppose that $dm(a+ib) = g(a)\,da\,db$ on $\Omega_{x_0,t}$. Write $u = Q_{x_0,t}(a)$. Since $\Omega_{x_0,t}$ has the form described in Theorem~\ref{thm:HHBrown}, the push-forward of $m$ by $Q_{x_0,t}$ has the form
		\begin{equation}
			\label{eq:pushUnique}
			4v_t(h_{x_0,t}^{-1}(a))g(a)\,da = 4v_t(h_{x_0,t}^{-1}(a))g(a)\frac{da}{du}du, \quad u\in Q_{x_0,t}(\Omega_{x_0,t}\cap\R).
		\end{equation}
		By the definition \eqref{eq:Qtmap} of $Q_{x_0,t}$ and Theorem~\ref{thm:HHBrown}, the density of $\mathrm{Brown}(x_0+i\sigma_t)$ has the form $(1/4\pi t)(du/da)$ that is strictly positive. 
		
		By Point 2 of Theorem~\ref{thm:HHpushforward}, taking $g(a) = (1/4\pi t)(du/da)$ to be the density of $\mathrm{Brown}(x_0+i\sigma_t)$ gives $\mathrm{Law}(x_0+\sigma_t)$; that is, $\mathrm{Law}(x_0+\sigma_t)$ has the form
		\[\frac{1}{\pi t}v_t(h_{x_0,t}^{-1}(a))\,du,\quad u\in Q_{x_0,t}(\Omega_{x_0,t}).\] 
		Since $du/da$ is positive, the only $g(a)$ that makes the measure in \eqref{eq:pushUnique} equal to $\mathrm{Law}(x_0+i\sigma_t)$ is $(1/4\pi t)(du/da)$. This shows that $\mathrm{Brown}(x_0+i\sigma_t)$ is the only measure  on $\overline{\Omega_{x_0,t}}$ that is absolutely continuous with respect to the Lebesgue measure such that the density is constant along vertical segments and the push-forward of $m$ by $Q_{x_0,t}$ is $\mathrm{Law}(x_0+\sigma_t)$.
	\end{proof}

	\section{The Brown measure computation}
	
	Let $y_{0}$ be a self-adjoint element, $\tilde{\sigma}_{s-\frac{t}{2}}$ and $\sigma_{\frac{t}{2}}$ be two semicircular elements, all freely
	independent. Denote the law of $y_0$ by $\nu$. We study the Brown measure of 
	\[
	y_{0}+\tilde{\sigma}_{s-\frac{t}{2}}+i\sigma_{\frac{t}{2}}
	\]
	with $0<\frac{t}{2}\leq s$. 
	
	If the law of $y_{0}+\tilde{\sigma}_{s-\frac{t}{2}}$ is a Dirac mass at one point, then the Brown measure of $y_{0}+\tilde{\sigma}_{s-\frac{t}{2}}+i\sigma_{\frac{t}{2}}$ is singular with respect to the Lebesgue measure on the plane, and is a semicircular distribution along a vertical segment. Thus, we recall our standing assumption (Assumption~\ref{assump:Standing}) that either $s>\frac{t}{2}$ or $\nu$ is not a Dirac mass, so that $\mathrm{Law}(y_{0}+\tilde{\sigma}_{s-\frac{t}{2}})$ is not a Dirac mass. 
	
	For convenience, we define 
	\[x_0=y_{0}+\tilde{\sigma}_{s-\frac{t}{2}}.\]
	By Theorem \ref{thm:HHBrown}, $\Omega_{x_0, t/2}$ is an open set of full measure of $\mathrm{Brown}(x_0+i\sigma_{t/2})$. Since $x_0+i\sigma_{t/2}$ depends on both parameters $s$ and $t$, we write
	\[\Omega_{s,t} = \Omega_{x_0, t/2}.\]
	We also write the boundary of $\Omega_{s,t}$ as $a+ib_{s,t}(a)$ instead of $a+ib_{x_0, t/2}(a)$. We recall from Remark~\ref{rem:Hextension} that given any $q$ in the support of $\mathrm{Law}(y_0+\sigma_s)$, $H_{y_0,s}^{-1}(q)$ means the unique point $a_0+iv_{y_0, s}(a_0)$ on the boundary of $\Lambda_{y_0,s}$.

	\subsection{The domain of the Brown measure\label{sect:domain}}

	By Theorem \ref{thm:BianeFC} and the definition of $\Omega_{x_0, t/2}$ (see Theorem \ref{thm:HHBrown}), the map
	\begin{equation}
	\label{eq:FunctionfEllipse}
	F_{s,t}(z)=H_{x_0, t/2}\circ H_{x_0, -t/2}^{-1}(z)
	\end{equation}
	is an injective conformal mapping from $(\overline{\Omega_{s,t}})^c$ to the complement of the support of $\mathrm{Law}(y_0+\sigma_s)$.
	
	We want to establish a push-forward result that the push-forward measure of $\mathrm{Brown}(y_0+c_s)$ by a map constructed by $H_{y_0,s-t}$ is $\mathrm{Brown}(y_0+\tilde{\sigma}_{s-t/2}+i\sigma_{t/2})$. The main theorem in this section establishes the connection between the domains $\Omega_{s,t}$ and $\Lambda_{y_0,s}$ of $\mathrm{Brown}(y_0+c_s)$ and $\mathrm{Brown}(y_0+\tilde{\sigma}_{s-t/2}+i\sigma_{t/2})$ respectively.  The strategy is to show that $F_{s,t}$, originally defined using $H_{x_0,t/2}$ and $H_{x_0,-t/2}$, can be written in terms of $H_{y_0,s}$ and $H_{y_0,s-t}$ as in Proposition~\ref{prop:Fst2steps}.  Figure~\ref{fig:domainmaps} demonstrates the connections of the complements of the supports of $\mathrm{Brown}(y_0+\tilde{\sigma}_{s-t/2}+i\sigma_{t/2})$, $\mathrm{Brown}(y_0+c_s)$, and $\mathrm{Law}(y_0+\sigma_{s})$, where $x_0 = y_0+\tilde{\sigma}_{s-t/2}$, by the holomorphic functions $F_{s,t}$, $H_{y_0,s-t}$ and $H_{y_0,s}$. We remark that the parameters $s$ and $t$ satisfy $0<t\leq 2s$; the parameter $s-t$ in the subscript of $H_{y_0,s-t}$ can be negative.

	\begin{theorem}
		\label{thm:TwoSteps}
		The function $H_{y_0, s-t}$ is an injective conformal map on $(\overline{\Lambda_{y_0,s}})^c$ and extends to a homeomorphism on $\Lambda_{y_0, s}^c$. We also have
		\begin{equation}
			\label{eq:Omegast}
			\Omega_{s,t}^c = H_{y_0, s-t}(\Lambda_{y_0, s}^c).
		\end{equation}
		In particular, $\Omega_{s,s}=\Lambda_{y_0,s}$, recovering the domain in Theorem \ref{thm:HZ}.
		\end{theorem}

		\begin{figure}
			\begin{center}
				\begin{tikzcd}
					&\mathrm{Law}(y_0+\sigma_{s}) \\[10pt]
					\mathrm{Brown}(y_0+c_s) \arrow{rr}{H_{y_0,s-t}} \arrow{ur}{H_{y_0,s}} & & \mathrm{Brown}(y_0+\tilde{\sigma}_{s-t/2}+i\sigma_{t/2}) \arrow[swap]{ul}{F_{s,t}} 
				\end{tikzcd}
			\caption{Holomorphic maps between the complements of the supports of  $\mathrm{Brown}(y_0+\tilde{\sigma}_{s-t/2}+i\sigma_{t/2})$, $\mathrm{Brown}(y_0+c_s)$, and $\mathrm{Law}(y_0+\sigma_{s})$\label{fig:domainmaps}}
		\end{center}
		\end{figure}

	\begin{proposition}
		\label{prop:Fst2steps}
		The inverse $F_{s,t}^{-1}$ of $F_{s,t}$ can be written as
		\begin{equation}
		\label{eq:fInvH}
		F_{s,t}^{-1}(z) = (H_{y_0, s-t}\circ H_{y_0, s}^{-1})(z)
		\end{equation}
		for all $z$ outside the support of $\mathrm{Law}(y_0+\sigma_s)$.
	\end{proposition}

This shows that ,when $y_0=0$, $F_{s,t}$ is the additive analogue of the function $f_{s,t}$ introduced in \cite{Ho2016} in the context of free Segal--Bargmann--Hall transform.
	\begin{proof}
		Recall that we denote $y_0+\tilde{\sigma}_{s-\frac{t}{2}}$ by $x_0$. By Theorem~\ref{thm:BianeFC},
		\begin{equation}
		G_{y_{0}+\sigma_{s}}\left(H_{x_0,t/2}(z)\right)=G_{x_0+\sigma_{t/2}}(H_{x_0,t/2}(z))=G_{x_{0}}(z)=G_{y_{0}%
			+\sigma_{s-t/2}}(z) \label{elliptic.subordation}%
		\end{equation}
		because $\tilde{\sigma}_{s-t/2}+\sigma_{t/2}$ has the same distribution as $\sigma_{s}$.
		When $|z|$ large, (\ref{elliptic.subordation}) becomes
		\begin{equation}
		H_{x_0,t/2}^{-1}(z)=K_{y_{0}+\sigma_{s-t/2}}(G_{y_{0}+\sigma_{s}}(z)). \label{H.inverse}%
		\end{equation}
		Since the $R$-transform of the sum of two freely independent variables is the
		sum of the $R$-transforms of each variable (See Section~\ref{sect:Rtransform}),
		\[
		R_{y_{0}+\sigma_{s-t/2}}(z)=R_{y_{0}}(z)+R_{\sigma_{s-t/2}}(z)=R_{y_{0}}(z)+\left(s-\frac{t}{2}\right) z.
		\]
		Substracting by $\frac{1}{z}$ gives us 
		\begin{equation}
		\label{eq:Ksminustover2}
		K_{y_{0}+\sigma_{s-t/2}}(z)=K_{y_{0}}(z)+\left(s-\frac{t}{2}\right) z.
		\end{equation}
		Therefore,
		\begin{equation}
		K_{y_{0}+\sigma_{s-t/2}}\Big(G_{y_{0}+\sigma_{s}}(z)\Big)=K_{y_{0}}(G_{y_{0}+\sigma_{s}}(z))+\left(s-\frac{t}{2}\right)  G_{y_{0}+\sigma_{s}}(z). \label{K.composite.G}%
		\end{equation}
		
		By the definition of $F_{s,t}^{-1}$ in~\eqref{eq:FunctionfEllipse},
		\begin{equation}
		\label{eq:fInvElliptic}
		F_{s,t}^{-1}(z)=H_{x_0,-t/2}\left(H_{x_0,t/2}^{\langle-1\rangle}(z)\right)
		=H_{x_0,t/2}^{\langle-1\rangle}(z)-\frac{t}{2}G_{x_{0}+\sigma_{s-t/2}%
		}\left(H_{x_0,t/2}^{-1}(z)\right)
		\end{equation}
		Using (\ref{H.inverse}) and (\ref{K.composite.G}), the above becomes
		\begin{equation}
		\label{eq:fInvCauchy}
		\begin{split}
		F_{s,t}^{-1}(z)=&K_{y_{0}+\sigma_{s-t/2}}(G_{y_{0}+\sigma_{s}}(z))-\frac{t}{2}
		G_{y_{0}+\sigma_{s}}(z)\\
		=&K_{y_{0}}(G_{y_{0}+\sigma_{s}}(z))+\left(s-\frac{t}{2}\right)
		G_{y_{0}+\sigma_{s}}(z)-\frac{t}{2}G_{y_{0}+\sigma_{s}}(z)\\
		=  &  K_{y_{0}}(G_{y_{0}+\sigma_{s}}(z))+(s-t) G_{y_{0}+\sigma_{s}}(z).
		\end{split}
		\end{equation}
		Now, since $H_{y_0, s}$ satisfies $G_{y_0+\sigma_{s}}(H_{y_0, s}(z)) = G_{y_0}(z)$, we have
		\[H_{y_0, s}^{-1}(z) = K_{y_0}(G_{y_0+\sigma_{s}}(z))\]
		for all large enough $|z|$. It follows from (\ref{eq:fInvCauchy}) that $F_{s,t}^{-1}$ can be written as
		\[
		F_{s,t}^{-1}(z) = H_{y_0, s}^{-1}(z) + (s-t) G_{y_0}(H_{y_0, s}^{-1}(z))= (H_{y_0, s-t}\circ H_{y_0, s}^{-1})(z)
		\]
		for all large enough $z$. Since both sides of the above expression are defined on the complement of the support of $\mathrm{Law}(y_0+\sigma_{s})$, (\ref{eq:fInvH}) holds for all $z$ in the complement of the support of $\mathrm{Law}(y_0+\sigma_{s})$ by analytic continuation.
	\end{proof}
\begin{proof}[{\bf Proof of Theorem \ref{thm:TwoSteps}}]
	The function $F_{s,t}^{-1 }$ is an injective conformal map on the complement of the support of $\mathrm{Law}(y_0+\sigma_s)$. Thus, by Proposition \ref{eq:fInvH}
	\[H_{y_0, s-t}(z) = F_{s,t}^{-1}\circ H_{y_0,s}(z), \quad z\in\Delta_{y_0, s}\]
	is an injective conformal map onto
	\[\{a+ib\in\C|\left\vert b\right\vert>b_{s,t}(a)\}.\]
	Now, that the function $H_{y_0, s-t}$ extends to a homeomorphism on $\overline{\Delta_{y_0,s}}$ follows from an elementary topological argument by regarding $\Delta_{y_0, s}\cup\{\infty\}$ and $\{a+ib\in\C|\left\vert b\right\vert>b_{x_0, t}(a)\}\cup\{\infty\}$ as two disks in the Riemann sphere. Thus, $H_{y_0, s-t}$ is an injective conformal map on $(\overline{\Lambda_{y_0, s}})^c$ and extends to a homeomophism on $\Lambda_{y_0, s}^c$ by Schwarz reflection about the real axis. 
	
	Equation \eqref{eq:Omegast} is a restatement of Proposition \ref{eq:fInvH}. If $s=t$, the holomorphic function $H_{y_0,s-t}$ is the identity map; therefore, $\Omega_{s,s} = \Lambda_{y_0,s}$ by  \eqref{eq:Omegast}.
	\end{proof}

\subsection{Two push-forward properties}
\label{sect:pushforward}

In Section~\ref{sect:domain}, we establish the connection between $\Lambda_{y_0,s}$ and $\Omega_{s,t}$ through the map $H_{y_0,s-t}$. In this section, we prove that the push-forward measure of $\mathrm{Brown}(y_0+c_s)$ by a canonical map constructed using $H_{y_0,s-t}$ is $\mathrm{Brown}(y_0+\tilde{\sigma}_{s-t/2}+i\sigma_{t/2})$. The main observation is that both $\mathrm{Brown}(y_0+c_s)$ and $\mathrm{Brown}(y_0+\tilde{\sigma}_{s-t/2}+i\sigma_{t/2})$ can be pushed forward to $\mathrm{Law}(y_0+\sigma_s)$, by Theorems~\ref{thm:HZ} and~\ref{thm:HHpushforward}. These push-forward maps are not injective; nevertheless, Proposition~\ref{prop:HHUnique} shows that they characterize $\mathrm{Brown}(y_0+c_s)$ and $\mathrm{Brown}(y_0+\tilde{\sigma}_{s-t/2}+i\sigma_{t/2})$.

For convenience, we use the notations $a+ib$ for the points in $\Omega_{s,t}$, $\alpha+i\beta$ for the points in $\Lambda_{y_0,s}$, and $u$ for the points in the support of $\mathrm{Law}(y_0+\sigma_s)$.

Define the function $a_{s,t}:\R\to\R$ by 
		\[a_{s,t}(\alpha) = \mathrm{Re}[H_{y_0,s-t}(\alpha+iv_{y_0,s}(\alpha))],\quad\alpha\in\R.\]
		Let $U_{s,t}:\Lambda_{y_0,s}\to\Omega_{s,t}$ be defined by
\begin{align}
	\mathrm{Re}\,U_{s,t}(\alpha+i\beta) &= a_{s,t}(\alpha)\nonumber\\
	\mathrm{Im}\,U_{s,t}(\alpha+i\beta) &= \frac{t\beta}{s}\label{eq:UDef}.
\end{align}
We will prove that $a_{s,t}$ is a homeomorphism on $\R$ in Proposition~\ref{prop:asthomeo}. We can then immediately see that $U_{s,t}$ is indeed a homeomorphism on the complex plane $\C$. In this section, we prove the following two push-forward properties that are introduced in Points 2 and 3 of Theorem~\ref{thm:summary1}.
\begin{theorem}
	\label{thm:pushforward}
	We have the following results about push-forward measures.
	\begin{enumerate}
		\item The push-forward of $\mathrm{Brown}(y_0+c_s)$ under the map $U_{s,t}$ is  $\mathrm{Brown}(y_0+\tilde{\sigma}_{s-t/2}+i\sigma_{t/2})$. 
		\item The push-forward of $\mathrm{Brown}(y_0+\tilde{\sigma}_{s-t/2}+i\sigma_{t/2})$ by the map
	\[Q_{s,t}(a+ib) = \frac{1}{s-t}[sa-t\alpha_{s,t}(a)]\]
	is $\mathrm{Law}(y_0+\sigma_s)$.
	\end{enumerate}
	\end{theorem}

	\begin{figure}
		\begin{center}
			\begin{tikzcd}
				&\mathrm{Law}(y_0+\sigma_{s}) \\[10pt]
				\mathrm{Brown}(y_0+c_s) \arrow{rr}{U_{s,t}} \arrow{ur}{\Psi_{y_0,s}} & & \mathrm{Brown}(y_0+\tilde{\sigma}_{s-t/2}+i\sigma_{t/2}) \arrow[swap]{ul}{Q_{s,t}} 
			\end{tikzcd}
		\caption{Push-forward relations between the probability measures $\mathrm{Brown}(y_0+\tilde{\sigma}_{s-t/2}+i\sigma_{t/2})$, $\mathrm{Brown}(y_0+c_s)$, and $\mathrm{Law}(y_0+\sigma_{s})$, where $x_0 = y_0+\tilde{\sigma}_{s-t/2}$.\label{fig:pushforward}}
	\end{center}
	\end{figure}
	Recall that the function $F_{s,t}$ is defined in \eqref{eq:FunctionfEllipse}. By Theorems~\ref{thm:HZ} and~\ref{thm:HHpushforward}, the push-forward of $\mathrm{Brown}(y_0+c_s)$ by $\Psi_{y_0,s}$ defined by
	\[\Psi_{y_0,s}(\alpha+i\beta) = H_{y_0,s}(\alpha+iv_{y_0,s}(\alpha)),\quad \alpha+i\beta\in\Lambda_{y_0,s}\]
	and the push-forward of $\mathrm{Brown}(y_0+\tilde{\sigma}_{s-t/2}+i\sigma_{t/2})$ by $Q_{x_0,t/2}$ (where $x_0=y_0+\tilde{\sigma}_{s-t/2}$) defined by
	\[Q_{x_0,t/2}(a+ib) = F_{s,t}(a+ib_{s,t}(a)),\quad a+ib\in\Omega_{s,t}\]
	are both $\mathrm{Law}(y_0+\tilde{\sigma}_{s-t/2}+i\sigma_{t/2})$. In the proof of Theorem~\ref{thm:pushforward}, we actually can see that $Q_{s,t} = Q_{x_0,t/2}$. Figure~\ref{fig:pushforward} illustrates the push-forward relations between all of these measures.

Before we prove this theorem, we first study the function $a_{s,t}$ in the definition of $U_{s,t}$.
\begin{proposition}
		\label{prop:asthomeo}
		The function $a_{s,t}$ is strictly increasing. It is a homeomorphism onto $\R$. In particular, $a_{s,t}$ has an inverse on $\R$ that is also strictly increasing. Furthermore, $a_{s,t}'(\alpha)>0$ for all $\alpha\in\Lambda_{y_0,s}\cap\R$. 
		
		The upper boundary curve $a+ib_{s,t}(a)$ of $\Omega_{s,t}$ can be parametrized by $\alpha\in\Lambda_{y_0,s}\cap\R$. The parameterization is
		\begin{equation}
			\label{eq:ParaBoundary}
			a+ib_{s,t}(a) = a_{s,t}(\alpha)+\frac{it}{s}v_{y_0,s}(\alpha).
		\end{equation}
	\end{proposition}
	\begin{proof}
		By a direct computation,
		\[a_{s,t}(\alpha) = \frac{s-t}{s}\left(\frac{t\alpha}{s-t}+\mathrm{Re}[H_{y_0,s}(\alpha+iv_{y_0,s}(\alpha))]\right).\]
		If $s>t$, then $a_{s,t}$ is strictly increasing because $\mathrm{Re}[H_{y_0,s}(\alpha+iv_{y_0,s}(\alpha))]$ is strictly increasing in $\alpha\in\R$ by Theorem \ref{thm:BianeFC}. If $s<t$, then we write
		\begin{equation*}
			a_{s,t}(\alpha) =\frac{t-s}{s}\left(\frac{(2s-t)\alpha}{t-s}+\mathrm{Re}[H_{y_0,-s}(\alpha+iv_{y_0,s}(\alpha))]\right)
		\end{equation*}
		which is a strictly increasing function since $\mathrm{Re}[H_{y_0,-s}(\alpha+v_{y_0,s}(\alpha))]$ is strictly increasing in $\alpha\in\R$, by Point 2 of Definition~\ref{def:H}. If $s=t$, $a_{s,t}$ is just the identity function. In any case, if $v_{y_0,s}(\alpha)>0$, $a_{s,t}$ is differentiable at $\alpha$ and $a_{s,t}'(\alpha)>0$ by Point 2 of Definition \ref{def:H}.
			
		By Theorem \ref{thm:TwoSteps}, $a+ib_{s,t}(a) = H_{s-t}(\alpha+iv_{y_0,s}(\alpha))$ for a unique $\alpha\in\Lambda_{y_0, s}\cap\R$. The imaginary part of $H_{s-t}(\alpha+iv_{y_0,s}(\alpha))$ is given by
		\[v_{y_0,s}(\alpha)\left(1-(s-t)\int\frac{1}{(\alpha-x)^2+v_{y_0,s}(\alpha)^2}\,d\nu(x)\right) = \frac{t}{s}v_{y_0,s}(\alpha)\]
		by \eqref{eq:IntOfSq}. This proves the parametrization \eqref{eq:ParaBoundary}.
	\end{proof}

	\begin{proposition}
	\label{prop:Ust}
	The function $U_{s,t}:\Lambda_{y_0,s}\to\Omega_{s,t}$ defined by \eqref{eq:UDef} is a diffeomorphism; it extends to a homeomorphism from $\overline{\Lambda_{y_0,s}}$ to $\overline{\Omega_{s,t}}$. Moreover, it agrees with $H_{y_0, s-t}$ on the boundary of $\Lambda_{y_0,s}$.
\end{proposition}
\begin{proof}
	By Point 1 of Theorem~\ref{thm:addellipse}, $a_{s,t}$ is injective, strictly increasing and differentiable in $\Lambda_{y_0, s}\cap\R$ with nonzero derivative; therefore, $U_{s,t}$ is a diffeomorphism from $\Lambda_{y_0,s}$ onto $\Omega_{s,t}$. Since $a_{s,t}$ is a homeomorphism defined on $\R$, the map $U_{s,t}$ can be extended to a homeomorhism in $\C$; in particular, it is a homeomorphism from $\overline{\Lambda_{y_0, s}}$ to $\overline{\Omega_{s,t}}$. 
	
	It is clear from \eqref{eq:ParaBoundary} that $U_{s,t}$ agrees with $H_{y_0,s-t}$ on the boundary of $\Lambda_{y_0,s}$.
\end{proof}

Before we prove Theorem~\ref{thm:pushforward}, we write the function $\alpha_{s,t}$ in Theorem~\ref{thm:addellipse} as the solution of the following integral equation
\begin{equation}
\label{eq:aandalpha}
a = a_{s,t}(\alpha_{s,t}(a)) = \alpha_{s,t}(a)+(s-t)\int\frac{(\alpha_{s,t}(a)-x)\,d\nu(x)}{(\alpha_{s,t}(a)-x)^2+v_{y_0,s}(\alpha_{s,t}(a))^2}.
\end{equation}

\begin{proof}[{\bf Proof of Theorem \ref{thm:pushforward}}]
	Recall that the density of $\mathrm{Brown}(y_0+c_s)$ is constant along vertial segments in $\Lambda_{y_0,s}$. By \eqref{eq:UDef}, the Jacobian matrix of $U_{s,t}$ on $\Lambda_{y_0,s}$ is diagonal and $\mathrm{Im}(U_{s,t}(\alpha+i\beta))$ depends linearly in $\beta$. Thus, the density of the push-forward measure of $\mathrm{Brown}(y_0+c_s)$ by $U_{s,t}$ is again constant along vertical segments in $\Omega_{s,t}$. 

	We apply Proposition~\ref{prop:HHUnique} to show that the push-forward of $\mathrm{Brown}(y_0+c_s)$ by $U_{s,t}$ is $\mathrm{Brown}(y_0+\tilde{\sigma}_{s-t/2}+i\sigma_{t/2})$. By Proposition~\ref{prop:Fst2steps}, for any $\alpha+i\beta\in \Lambda_{y_0,s}$,
	\begin{align*}
		Q_{x_0,t}\circ U_{s,t}(\alpha+i\beta) &= F_{s,t}(a_{s,t}(\alpha)+ib_{s,t}(a_{s,t}(\alpha)))\\
		&=H_{y_0,s}(\alpha+iv_{y_0,s}(\alpha))\\
		&=\Psi_{y_0,s}(\alpha+i\beta).
	\end{align*}
	This shows that if we further push forward by $Q_{x_0,t}$ the push-forward of $\mathrm{Brown}(y_0+c_s)$ by $U_{s,t}$, we get the push-forward of $\mathrm{Brown}(y_0+c_s)$ by $\Psi_{y_0,s}$, which is $\mathrm{Law}(y_0+\sigma_{s})$ by Theorem~\ref{thm:HZ}. This completes the proof of Point 1 of the theorem.

	We now prove Point 2. By Point 1, $\mathrm{Brown}(y_0+\tilde{\sigma}_{s-t/2}+i\sigma_{t/2})$ is the push-forward measure of $\mathrm{Brown}(y_0+c_s)$. Since $U_{s,t}$ is a diffeomorphism on $\Lambda_{y_0,s}$, the push-forward of $\mathrm{Brown}(y_0+\tilde{\sigma}_{s-t/2}+i\sigma_{t/2})$ by $\Psi_{y_0,s}\circ U_{s,t}^{-1}$ is $\mathrm{Law}(y_0+\sigma_{s})$. (In fact, by the proof of Point 1, $\Psi_{y_0,s}\circ U_{s,t}^{-1} = Q_{x_0,t}$.) We then compute

	\begin{align*}
		\Psi_{y_0,s}\circ U_{s,t}^{-1}(a+ib) &= \Psi_{y_0,s}\left(\alpha_{s,t}(a)+i\frac{s}{t}b\right)\\
		&= \alpha_{s,t}(a)+ s\int\frac{\alpha_{s,t}(a)-x}{(\alpha_{s,t}(a)-x)^2+v_{y_0,s}(\alpha_{s,t}(a))}d\nu(x)\\
		&= \alpha_{s,t}(a) +\frac{s}{s-t}(a-\alpha_{s,t}(a))
	\end{align*}
	where we use \eqref{eq:aandalpha} in the last equality. The above equation simplies to the definition of $Q_{s,t}$, completing the proof.
	\end{proof}

The density $w_{y_0,s,t}$ of $\mathrm{Brown}(y_0+\tilde{\sigma}_{s-t/2}+i\sigma_{t/2})$ can be computed in terms of the density $w_{y_0,s}$ of $\mathrm{Brown}(y_0+c_s)$. We will give an alternative formula in the next section.
\begin{corollary}
	\label{cor:stFormulas}
	Let $r= t/s$ and write $a+ib = U_{s,t}(\alpha+i \beta)$ for all $\alpha+i\beta\in\Lambda_{y_0,s}$. Then we have
	\[w_{y_0,s,t}(a+ib) = \frac{1}{r}\frac{w_{y_0,s}(\alpha+i\beta)}{r+2\pi(1-r)s\cdot w_{y_0,s}(\alpha+i\beta)}\]
	for all $a+ib\in\Omega_{s,t}$.
\end{corollary}
\begin{proof}
	Denote $r = t/s$. We can write the function $a_{s,t}(\alpha)$ defined in Proposition~\ref{prop:asthomeo} as
	\begin{align*}
	a_{s,t}(\alpha)& = \alpha+(1-r)s\,\mathrm{Re}\left[\int\frac{d\nu(x)}{\alpha+iv_{y_0,s}(\alpha)-x}\right]\\
	&= \alpha+(1-r)[H_{y_0,s}(\alpha+iv_{y_0,s}(\alpha))-\alpha]\\
	&= (1-r)\psi_{y_0,s}(\alpha)+r\alpha.
	\end{align*}
	So, we have
	\[\frac{d a_{s,t}(\alpha)}{d\alpha} = r+2\pi(1-r)s \cdot w_{y_0,s} (\alpha+i\beta).\]
	By Theorem \ref{thm:pushforward}, we can compute the density $w_{y_0, s,t}(a+ib)\,da\,db$ in terms of $w_{y_0,s}$ as
	\begin{align*}
	w_{y_0, s,t}(a+ib)\,da\,db &= w_{y_0,s}(\alpha+i\beta)\,d\alpha\,d\beta\\
	&= w_{y_0,s}(\alpha+i\beta)\,\frac{d\alpha}{da}\,\frac{d\beta}{db}\,da\,db\\
	&= \frac{1}{r}\frac{w_{y_0,s}(\alpha+i\beta)}{r+2\pi(1-r)s\cdot w_{y_0,s}(\alpha+i\beta)}\,da\,db,
	\end{align*}
	completing the proof.
\end{proof}

	\subsection{The density of the Brown measure}
	\label{sect:density}
	The main theorem of this section is to compute the density of $\mathrm{Brown}(y_0+\tilde{\sigma}_{s-t/2}+i\sigma_{t/2})$ stated in Point 1 of Theorem~\ref{thm:summary1}. 

	\begin{theorem}
		\label{thm:addellipse}
		The Brown measure of $y_0+\tilde{\sigma}_{s-t/2}+i\sigma_{t/2}$ is absolutely continuous with respect to the Lebesgue measure on the plane and is supported on $\overline{\Omega_{s,t}}$. The open set $\Omega_{s,t}$ is a set of full measure of the Brown measure. The density of the Brown measure is given by
		\[w_{y_0,s,t}(a+ib) = \frac{1}{2\pi t}\left(1+t\frac{d}{da}\int\frac{\alpha_{s,t}(a)-x}{(\alpha_{s,t}(a)-x)^2+v_{y_0,s}(\alpha_{s,t}(a))^2}\,d\nu(x)\right)\]
		on the set $\Omega_{s,t}$. In particular, the density is constant along the vertical segments.
	\end{theorem}
	\begin{proof}
		We only need to compute the density. The proof uses the first push-forward property stated in Theorem~\ref{thm:pushforward}. By Theorem~\ref{thm:HZ}, $\mathrm{Brown}(y_0+c_s)$ is given by
		\begin{align*}
			&\frac{1}{2\pi s}\frac{d}{d\alpha}H_{y_0,s}(\alpha+iv_{y_0,s}(\alpha))\,d\alpha\,d\beta\\
			&=\frac{1}{2\pi s}\frac{d}{d\alpha}\left(a_{s,t}(\alpha)+t\int\frac{(\alpha-x)\,d\nu(x)}{(\alpha-x)^2+v_{y_0,s}(\alpha)^2}\right)d\alpha\,d\beta
		\end{align*}
		for $\alpha+i\beta\in\Lambda_{y_0,s}$. The determinant of the Jacobian matrix of $U_{s,t}$ defined in \eqref{eq:UDef} is $(t/s)(da_{s,t}/d\alpha)$. By the push-forward property in Point 1 of Theorem~\ref{thm:pushforward}, we compute $\mathrm{Brown}(y_0+\tilde{\sigma}_{s-t/2}+i\sigma_{t/2})$ by doing a change of variable $a+ib = a_{s,t}(\alpha)+i(t/s)\beta$ to the above formula of $\mathrm{Brown}(y_0+c_s)$ and get
		\begin{align*}
			\mathrm{Brown}(y_0+\tilde{\sigma}_{s-t/2}+i\sigma_{t/2}) = \frac{1}{2\pi t}\frac{d}{da}\left(a+t\int\frac{(\alpha_{s,t}(a)-x)\,d\nu(x)}{(\alpha_{s,t}(a)-x)^2+v_{y_0,s}(\alpha_{s,t}(a))^2}\right)da\,db
		\end{align*}
		on $\Omega_{s,t}$. We have completed the proof.
	\end{proof}

	Before we end this section, we prove Proposition~\ref{prop:IntroEq} in the following corollary. 
	
	\begin{corollary}
		\label{cor:IntroEq}
		Given any $a\in\R$, \eqref{eq:IntroEq} has a pair of solution $\alpha\in\R$ and $v>0$ if and only if $a\in\Omega_{s,t}\cap\R$. In this case, the solution is unique; moreover, we have $\alpha_{s,t}(a)=\alpha$, $v_{y_0,s}(\alpha_{s,t}(a))=v$ and $b_{s,t}(a) = \frac{t}{s}v$.
	\end{corollary}
	\begin{proof}
		Let $a\in\Omega_{s,t}\cap\R$. Then, by \eqref{eq:IntOfSq} and \eqref{eq:aandalpha}, $\alpha = \alpha_{s,t}(a)$ and $v = v_{y_0,s}(\alpha_{s,t}(a))$ is a pair of solution of \eqref{eq:IntroEq}. This shows existence of the equation. We now show the solution is indeed unique. Suppose that $\alpha\in\R$ and $v>0$ is a pair of solution. We must show that $\alpha = \alpha_{s,t}(a)$ and $v = v_{y_0,s}(\alpha_{s,t}(a))$. By \eqref{eq:IntOfSq}, the first equation of \eqref{eq:IntroEq} says $v = v_{y_0,s}(\alpha)$. Using the first equation
		\[\int\frac{d\nu(x)}{(\alpha-x)^2+v^2}=\frac{1}{s},\] 
		of \eqref{eq:IntroEq}, the second equation of \eqref{eq:IntroEq} can be written as
		\[	a = \alpha+(s-t)\int\frac{(\alpha-x)\,d\nu(x)}{(\alpha-x)^2+v_{y_0,s}(\alpha)^2},\]
		which shows $a = a_{s,t}(\alpha)$, and so $\alpha = \alpha_{s,t}(a)$.

		Conversely, suppose that \eqref{eq:IntroEq} has a pair of solution $\alpha\in\R$ and $v>0$. Then the arguemnt that shows uniqueness of solution in the preceding paragraph proves that $v = v_{y_0,s}(\alpha_{s,t}(a))$ and so $a = a_{s,t}(\alpha)$. Thus, \eqref{eq:ParaBoundary} shows $b_{s,t}(a) = tv/s>0$, and so $a\in\Omega_{s,t}\cap\R$.
	\end{proof}

\section{Asymptotic behaviors of adding a circular element\label{sect:circularAsymp}}
\subsection{The graph of $v_{y_0, s}$ as $s\to\infty$}
In this section, we study the asymptotic behavior of $v_{y_0, s}$ and $\Lambda_{y_0, s}$ as $s\to\infty$. Below is the main theorem of this section.
\begin{theorem}
	\label{thm:vsAsymp}
	The following asymptotic behaviors of the graph of $v_{y_0, s}$ hold.
	\begin{enumerate}
		\item Let $D_\nu = \sup\{\left\vert x-y\right\vert| x, y\in\mathrm{supp}\,\mu\}$. When $s\geq 4 D_\nu^2$, the function $v_{y_0,s}$ is unimodal.  In particular, $\Lambda_{y_0,s}\cap\R$ is an interval.
		\item 		Given any $c>1$, we have
		\[\left\vert\sup \Lambda_{y_0,s}\cap\R-(\tau(y_0)+\sqrt{s})\right\vert<\frac{3c\tau(y_0^2)}{2\sqrt{s}}\] 
		and
		\[\left\vert\inf \Lambda_{y_0,s}\cap\R-(\tau(y_0)-\sqrt{s})\right\vert<\frac{3c\tau(y_0^2)}{2\sqrt{s}}\] 
		for all large enough $s$. In particular,
		\[ \Lambda_{y_0,s}\cap\R\subset \left(\tau(y_0)-\sqrt{s}-\frac{3c\tau(y_0^2)}{2\sqrt{s}}, \tau(y_0)+\sqrt{s}+\frac{3c\tau(y_0^2)}{2\sqrt{s}}\right)\]
		for all large enough $s$.
		\item Given any $\varphi_0\in(0,\pi/2)$, then for all large enough $s$, for all $\left\vert \cos\varphi\right\vert\leq \cos\varphi_0$, the unique $\alpha\in\R$ such that
		\[H_{y_0,s}(\alpha+iv_{y_0,s}(\alpha)) = 2\sqrt{s}\cos\varphi.\]
		satisfies
		\[\left\vert \alpha+iv_{y_0,s}(\alpha)-\sqrt{s}e^{i\varphi}\right\vert<\frac{1}{(\sin\varphi_0)\sqrt{s}}.\]	
	\end{enumerate}
\end{theorem}
Point 1 of Theorem~\ref{thm:vsAsymp} is a known result in \cite[Theorem 3.2]{HasebeUeda2018}. We state it here for completeness; it is also useful for us to understand the asymptotic behaviors of $\Lambda_{y_0,s}$.

We study the asymptotic behaviors of $v_{y_0, s}$ by looking at $v_{\frac{y_0}{\sqrt{s}}, 1}$, whose graph is scaled by $\sqrt{s}$ the graph of $v_{y_0, s}$. We look at 
\[H_{\frac{y_0}{\sqrt{s}},1}(z)=z+G_{\frac{y_0}{\sqrt{s}}}(z).\]
If $s$ is large enough, $H_{\frac{y_0}{\sqrt{s}},1}$ is defined for all $|z|>\frac{1}{2}$ since $y_0$ is assumed to be bounded.

We assume $y_0$ is centered and has unit variance until the proof of Theorem~\ref{thm:vsAsymp} for simplicity. The function $H_{\frac{y_0}{\sqrt{s}},1}$ is the inverse subordination function of the free convolution $\frac{y_0}{\sqrt{s}}+\sigma_1$. When $s$ is large,  $\frac{y_0}{\sqrt{s}}+\sigma_1$ behaves like $\sigma_1$; our strategy is to compare $\frac{y_0}{\sqrt{s}}+\sigma_1$ with $\sigma_1$. Denote by $k(z)$ the function $H_{0,1}(z)$; that is
\[k(z) = z+\frac{1}{z}.\]
The techniques in this section are similar to techniques in proving the supercovergence results in \cite{BercoviciVoiculescu1995, BercoviciWangZhong2018, Wang2010}.
\begin{lemma}
	\label{lem:Hkdist}
	Assume $y_0$ is a bounded random variable with $\tau(y_0)=0$ and $\tau(y_0^2)=1$. Then given any $c>1$, there exists $s_0>0$ such that
	\[\left\vert H_{\frac{y_0}{\sqrt{s}},1}(z) -k(z)\right\vert< \frac{c}{s|z|^3},\quad |z|>\frac{1}{2}\]
	for all $s\geq s_0$.
	\end{lemma}
\begin{proof}
When $s$ is large enough, we can write
\begin{align*}
H_{\frac{y_0}{\sqrt{s}},1}(z)  
&= k(z)+\frac{1}{s}\sum_{n=2}^\infty\frac{\tau(y_0^n)}{s^{\frac{n}{2}-1}z^{n+1}}
\end{align*}
for all $|z|>\frac{1}{2}$. Observe that
\[\left\vert\sum_{n=2}^\infty\frac{\tau(y_0^n)}{s^{\frac{n}{2}-1}z^{n+1}} \right\vert\leq \frac{\tau(y_0^2)}{|z|^3}+\frac{1}{|z|^3}\sum_{n=3}^\infty\frac{\left\vert\tau(y_0^n)\right\vert}{s^{\frac{n}{2}-1}(1/2)^{n-2}}\]
for all $|z|>\frac{1}{2}$. Since we assume $\tau(y_0^2)=1$ and
\[\lim_{s\to\infty}\sum_{n=3}^\infty\frac{\left\vert\tau(y_0^n)\right\vert}{s^{\frac{n}{2}-1}(1/2)^{n-2}}= 0,\]
the result follows.
\end{proof}
We compute that $k'(z) = 1-\frac{1}{z^2}$; the double zeros of $k$ are $1$ and $-1$.  The next lemma shows that $H_{\frac{y_0}{\sqrt{s}},1}$ also has doubles zeros at a point close to $1$ and a point close to $-1$. Since $v_{\frac{y_0}{\sqrt{s}},1}$ is unimodal for large $s$, these two points are the only double zeros of~$H_{\frac{y_0}{\sqrt{s}},1}$. Since $H_{\frac{y_0}{\sqrt{s}},1}$ is symmetric about the real axis, these two double zeros must be real numbers. Again since $v_{\frac{y_0}{s},1}$ is unimodal for large $s$, $\Lambda_{\frac{y_0}{s},1}\cap\R$ is an open interval and the two double zeros of $H_{\frac{y_0}{\sqrt{s}},1}$ are the endpoints of~$\Lambda_{\frac{y_0}{s},1}\cap\R$.
\begin{lemma}
	\label{lem:zeros}
	Given any $c>1$, there exists $s_0$ such that
	\[\left\vert H_{\frac{y_0}{\sqrt{s}},1}'(\pm1+re^{i\theta})-k'(\pm1+re^{i\theta})\right\vert<  \frac{3c}{s(1-r)^4}\]
	for all $s\geq s_0$ and $r<\frac{1}{2}$.	
	\end{lemma}
\begin{proof}
	Recall that
	\[H_{\frac{y_0}{\sqrt{s}},1}(z) = k(z)+\frac{1}{s}\sum_{n=2}^\infty\frac{\tau(y_0^n)}{s^{\frac{n}{2}-1}z^{n+1}};\]
	we compute
	\begin{equation}
	\label{eq:H'k'}
	H_{\frac{y_0}{\sqrt{s}},1}'(z) = 1-\frac{1}{z^2}-\frac{1}{s}\left(\frac{3\tau(y_0^2)}{z^4}+\frac{1}{z^4}\sum_{n=3}^\infty\frac{(n+1)\tau(y_0^n)}{s^{\frac{n}{2}-1}z^{n-2}}\right)
	\end{equation}
	
	Let $c>1$ be given. If $z = 1+r e^{i\theta}$ with $r<1/2$, then for all large enough $s$,
	\[\left\vert\frac{3\tau(y_0^2)}{z^4}+\frac{1}{z^4}\sum_{n=3}^\infty\frac{(n+1)\tau(y_0^n)}{s^{\frac{n}{2}-1}z^{n-2}}\right\vert<\frac{3c}{(1-r)^4}\]
	since $\left\vert z\right\vert > 1-r>1/2$ and $\tau(y_0^2)=1$. The case for $z = 1-re^{i\theta}$ is similar.
	\end{proof}

	\begin{proposition}
		\label{prop:vEndPoint}
		We have
			\[1-\frac{3c}{2s}<\sup \Lambda_{\frac{y_0}{\sqrt{s}},1}\cap\R<1+\frac{3c}{2s}\] 
		and
		\[-1-\frac{3c}{2s}<\inf \Lambda_{\frac{y_0}{\sqrt{s}},1}\cap\R<-1+\frac{3c}{2s}\] 
		for all large enough $s$. In particular,
		\[ \Lambda_{\frac{y_0}{\sqrt{s}},1}\cap\R\subset \left(-1-\frac{3c}{2s}, 1+\frac{3c}{2s}\right)\]
		for all large enough $s$.
		\end{proposition}
	\begin{proof}
		Recall that $\sup\Lambda_{\frac{y_0}{s},1}\cap\R$ and $\inf\Lambda_{\frac{y_0}{s},1}\cap\R$ are the only double zeros for $H_{\frac{y_0}{\sqrt{s}},1}$ when $s$ is large enough so that $v_{y_0,s}$ is unimodal.
		
		 Let $c>1$. 
		We compute, with $z = 1+re^{i\theta}$,
	\[\left\vert1-\frac{1}{z^2}\right\vert = \left\vert\frac{r(2e^{i\theta}+r e^{2i\theta})}{(1+r e^{i\theta})^2}\right\vert\geq\frac{r(2-r)}{(1+r)^2}.\]
	Then, by choosing any $1<c'<c$ in Lemma \ref{lem:zeros}, $r=\frac{3c}{2s}$ satisfies
	\[\left\vert H_{\frac{y_0}{\sqrt{s}},1}'(1+re^{i\theta})-k'(1+re^{i\theta})\right\vert<  \frac{3c'}{s(1-r)^4} <\frac{r(2-r)}{(1+r)^2}\leq\left\vert 1-\frac{1}{z^2}\right\vert\]
	for all large enough $s$, because, if $s$ is large enough
	\[\frac{3c'(1+r)^2}{r(2-r)(1-r)^4}=\frac{3c'(1+r)^2 2s}{3c(2-r)(1-r)^4}<s.\]
	By Rouch\'e's theorem, we have
	\[1-\frac{3c}{2s}<\sup \Lambda_{\frac{y_0}{\sqrt{s}},1}\cap\R<1+\frac{3c}{2s}.\] 
	The proof of
	\[-1-\frac{3c}{2s}<\inf \Lambda_{\frac{y_0}{\sqrt{s}},1}\cap\R<-1+\frac{3c}{2s}\] 
	 is similar. 
	\end{proof}

\begin{proposition}
	\label{prop:vUpper}
	Given any $\varphi_0\in(0,\pi/2)$, then for all large enough $s$, for all $\left\vert \cos\varphi\right\vert\leq \cos\varphi_0$, the unique $\alpha\in\R$ such that
	\[H_{\frac{y_0}{\sqrt{s}},1}(\alpha+iv_{\frac{y_0}{\sqrt{s}},1}(\alpha)) = 2\cos\varphi.\]
	satisfies
	\[\left\vert \alpha+iv_{\frac{y_0}{\sqrt{s}},1}(\alpha)-e^{i\varphi}\right\vert<\frac{1}{(\sin\varphi_0)s}.\]	
	\end{proposition}
\begin{proof}
	 Fix $\varphi_0 \in(0,\pi/2) $ and let $r=\frac{1}{(\sin\varphi_0) s}$. Then, given any $\varphi\in(0,\pi)$ such that $\sin\varphi\geq\sin\varphi_0$, we have, for large $s$,
	\begin{equation}
	\label{eq:kLower}
	\begin{split}
		\left\vert k(e^{i\varphi}+re^{i\theta})-k(e^{i\varphi})\right\vert& =\left\vert re^{i\theta}\left(\frac{re^{i\theta}+2i\sin\varphi}{e^{i\varphi}+re^{i\theta}}\right)\right\vert\\
		&\geq\frac{1}{\sin\varphi_0 s}\frac{2\sin\varphi_0-r}{1+r}.
		\end{split}
		\end{equation}
	Fix any $1<c<2$. The lower bound in~\eqref{eq:kLower} of $s\left\vert k(e^{i\varphi}+re^{i\theta})-k(e^{i\varphi})\right\vert$ converges to $2$ as $s\to\infty$. It follows from Lemma~\ref{lem:Hkdist} that, for all large enough $s$,
	\begin{align*}
		\left\vert H_{\frac{y_0}{\sqrt{s}},1}(e^{i\varphi}+re^{i\theta})-k(e^{i\varphi}+re^{i\theta})\right\vert&<\frac{c}{s(1-r)^3}\\
		&<\left\vert k(e^{i\varphi}+re^{i\theta})-k(e^{i\varphi})\right\vert\\
		&= \left\vert k(e^{i\varphi}+re^{i\theta})-2\cos\varphi)\right\vert;
		\end{align*}
	by Rouche's theorem, there exists a point $p_{\cos\varphi}$ such that $\left\vert p_{\cos\varphi}-e^{i\varphi}\right\vert<\frac{1}{(\sin\varphi_0) s}$ and
	\[H_{\frac{y_0}{\sqrt{s}},1}(p_{\cos\varphi}) = 2\cos\varphi.\]
	In particular, $H_{\frac{y_0}{\sqrt{s}},1}(p_{\cos\varphi}) \in\R$. The proposition now follows from the fact that $v_{\frac{y_0}{\sqrt{s}},1}(\alpha)$ is the unique positive number (if exists) such that
	\[H_{\frac{y_0}{\sqrt{s}},1}(\alpha+iv_{\frac{y_0}{\sqrt{s}},1}(\alpha)) \in\R.\]
	This completes the proof.
	\end{proof}
\begin{proof}[{\bf Proof of Theorem \ref{thm:vsAsymp}}]
	Point 1 is a result in \cite[Theorem 3.2]{HasebeUeda2018} which states that $v_s$ is unimodal for $s\geq 4 D_\nu^2$. This implies $\Lambda_{y_0,s}\cap\R = (\inf\Lambda_{y_0,s}, \sup\Lambda_{y_0,s})$ is an interval.
	
	Let 
	\[Y = \frac{y_0-\tau(y_0)}{\sqrt{\tau(y_0^2)}}\]
	and write $t = s/\tau(y_0^2)$.	By Theorem \ref{thm:HZ}, $\Lambda_{y_0, s}$ is the domain of full measure of $\mathrm{Brown}(y_0+c_s)$. Since $\mathrm{Brown}(y_0+c_s)$ is the push-forward of $\mathrm{Brown}\left(\frac{Y}{\sqrt{t}}+c_1\right)$ by the function 
	\[z\mapsto \tau(y_0)+z\sqrt{t\tau(y_0^2)} = \tau(y_0)+z\sqrt{s}\]
	 by \cite[Proposition 2.14]{HaagerupSchultz2007}. Thus, 
	\[\Lambda_{y_0, s} = \left\{\left.\tau(y_0)+z\sqrt{s}\in\C\right| z\in \Lambda_{\frac{Y}{\sqrt{t}},1}\right\}.\]
	Points 2  and 3 then follow from applying Proposition \ref{prop:vEndPoint} and Proposition \ref{prop:vUpper} with $t=s/\tau(y_0^2)$ in place of $s$ respectively; $\Lambda_{y_0, s}$ is obtained by scaling $\Lambda_{\frac{Y}{\sqrt{t}},1}$ by $\sqrt{s}$ and translating by $\tau(y_0)$.
	\end{proof}
\subsection{The density as $s\to\infty$}
In this section, we estimate the density of $\mathrm{Brown}(y_0+c_s)$ for large $s$. The Brown measure of $c_s$ is the uniform measure on the disk of radius $\sqrt{s}$; that is, the density is the constant
\begin{equation}
	\label{eq:cdensity}
	\frac{1}{\pi s}
\end{equation}
inside the unit disk. The following theorem states that for a fixed $y_0$, as $s\to\infty$, the density $w_{y_0,s}$ of $\mathrm{Brown}(y_0+c_s)$ is approximately the same constant in \eqref{eq:cdensity}.
\begin{theorem}
	\label{thm:CircularDensity}
	Denote by $w_{y_0, s}$ the density of $\mathrm{Brown}(y_0+c_s)$. Then, for any $c>1$ and  $\varphi_0\in(0,\pi/2)$, we have
	\[\left\vert w_{y_0,s}(\alpha+i\beta) - \frac{1}{\pi s}\right\vert < \frac{c\tau(y_0^2)}{2\pi s^2\sin^2\varphi_0}\left(3+\frac{2}{\sin\varphi_0}\right),\quad \left\vert\psi_{y_0,s}(\alpha)\right\vert < 2\sqrt{s}\cos\varphi_0\]
	for all large enough $s$.
\end{theorem}
To simplify the computation, we assume $\tau(y_0) = 0$ and $\tau(y_0^2)= 1$ until the proof of the theorem. The key is to estimate the difference between the complex derivatives $H_{\frac{y_0}{\sqrt{s}},1}'$ and $k'$; indeed the density is directly related to the real part of the complex derivative of the subordination function $H_{\frac{y_0}{\sqrt{s}},1}^{-1}$.
\begin{lemma}
	\label{lem:k'Est}
	Given any $c>1$ and $\varphi_0\in(0,\pi/2)$, for all sufficient large $s$, the unique $\alpha$ such that
	\[H_{\frac{y_0}{\sqrt{s}},1}(\alpha+iv_{\frac{y_0}{\sqrt{s}},1}(\alpha)) = 2\cos\varphi,\quad \sin\varphi>\sin\varphi_0\]
	satisfies
	\[\left\vert\frac{1}{\mathrm{Re}(1/k'(\alpha+iv_{\frac{y_0}{\sqrt{s}},1}(\alpha)))}-\frac{1}{\mathrm{Re}(1/k'(e^{i\varphi}))}\right\vert<\frac{2c}{s\sin^3\varphi_0}.\]
\end{lemma}
\begin{proof}
	Fix any $\varphi_0\in(0,\pi/2)$ and $c>1$. By Proposition~\ref{prop:vUpper}, for any $\varphi\in(0,\pi)$ such that $\sin\varphi > \sin\varphi_0$, the unique $\alpha\in\R$ such that
	\[H_{\frac{y_0}{\sqrt{s}},1}(\alpha+iv_{\frac{y_0}{\sqrt{s}},1}(\alpha)) = 2\cos\varphi.\]
	satisfies
	\begin{equation}
	\label{eq:alphavarphi}
	\left\vert \alpha+iv_{\frac{y_0}{\sqrt{s}},1}(\alpha)-e^{i\varphi}\right\vert<\frac{1}{(\sin\varphi_0)s}
	\end{equation}
	for all large enough $s$. We know that $\frac{1}{\mathrm{Re}(1/k'(z))} = 2$ because
	\begin{equation}
	\label{eq:k'unit}
	\frac{1}{k'(z)} = \frac{e^{i\varphi}}{e^{i\varphi}-e^{-i\varphi}}=\frac{1}{2}(1-i\cot\varphi).
	\end{equation} 
	Using~\eqref{eq:alphavarphi} and~\eqref{eq:k'unit}, we have
	\begin{equation}
	\label{eq:mvt}
	\frac{1}{(1/2-\left\vert\mathrm{Re}(1/k'(w))-\mathrm{Re}(1/k'(e^{i\varphi}))\right\vert)^2} < 4\sqrt{c}
	\end{equation}
	for all large enough $s$.
	
	Write $z = e^{i\varphi}$ and $w = \alpha+iv_{\frac{y_0}{\sqrt{s}},1}(\alpha)$. Observe that
	\begin{equation}
	\label{eq:1/k'Est}
	\frac{1}{k'(w)}-\frac{1}{k'(z)}=\frac{w^2}{w^2-1} - \frac{z^2}{z^2-1} = \frac{(z-w)(z+w)}{(w^2-1)(z^2-1)}.
	\end{equation}
		Also, it is straightforward to check that $\left\vert z^2-1\right\vert=\left\vert e^{2i\varphi}-1\right\vert = 2\sin\varphi$, and, by~\eqref{eq:alphavarphi},
	\[\left\vert w^2-z^2\right\vert=\left\vert w-z\right\vert \left\vert w+z\right\vert <\frac{1}{(\sin\varphi_0)s}\left(2+\frac{1}{(\sin\varphi_0)s}\right).\]
	We have, for all large enough $s$,
	\[\left\vert \frac{1}{k'(w)}-\frac{1}{k'(z)}\right\vert < \frac{1}{4\sin^2\varphi_0}\frac{2\sqrt{c}}{s(\sin\varphi_0)}.\]
	Thus, by the mean value theorem (applied to the function $1/(\frac{1}{2}+x)$), and \eqref{eq:alphavarphi}-\eqref{eq:1/k'Est},
	\begin{align*}
	\left\vert\frac{1}{\mathrm{Re}(1/k'(w))}-\frac{1}{\mathrm{Re}(1/k'(e^{i\varphi}))}\right\vert&\leq \frac{\left\vert\mathrm{Re}(1/k'(w))-\mathrm{Re}(1/k'(e^{i\varphi}))\right\vert}{(1/2-\left\vert\mathrm{Re}(1/k'(w))-\mathrm{Re}(1/k'(e^{i\varphi}))\right\vert)^2}\\
	&<4\sqrt{c}\frac{\sqrt{c}}{2s\sin^3\varphi_0} =\frac{2c}{s\sin^3\varphi_0}
	\end{align*}
	for all large enough $s$, completing the proof.
	\end{proof}

\begin{lemma}
	\label{lem:H'k'Diff}
	For any $c>1$, we have
	\[\left\vert\frac{1}{\mathrm{Re}(1/H_{\frac{y_0}{\sqrt{s}},1}'(z))}-\frac{1}{\mathrm{Re}(1/k'(z))}\right\vert<\frac{3c}{s\left\vert z\right\vert^4}\frac{1}{[\mathrm{Re}(1/k'(z))]^2}\frac{1}{\left\vert k'(z)\right\vert^2},\quad \left\vert z\right\vert>\frac{1}{2}.\]
	for all large enough $s$.
	\end{lemma}
When $\left\vert z\right\vert = 1$ but $z\neq 1, -1$, the right hand side of the inequality does not divide by zero. More explicitly, if $z = e^{i\varphi}$, we have
\begin{equation}
\label{eq:k'=sin}
\left\vert k'(z)\right\vert = \left\vert z^2 -1\right\vert = 2\sin\varphi.
\end{equation}
\begin{proof}
	Let $c>1$. By~\eqref{eq:H'k'}, for all $|z|>\frac{1}{2}$,
	\begin{equation}
	\label{eq:H'k'Est}
	\left\vert H_{\frac{y_0}{\sqrt{s}},1}'(z)  - k'(z)\right\vert\leq \frac{1}{s\left\vert z\right\vert^4}\left(3\tau(y_0^2)+\sum_{n=3}^\infty\frac{(n+1)\left\vert\tau(y^n)\right\vert}{s^{\frac{n}{2}-1}(1/2)^{n-2}}\right)<\frac{3c^{1/3}\tau(y_0^2)}{s\left\vert z\right\vert^4}
	\end{equation}
	for all large enough $s$.  We then must have
	\[\left\vert\frac{1}{\mathrm{Re}(1/H_{\frac{y_0}{\sqrt{s}},1}'(z))}\right\vert < \frac{c^{1/3}}{\mathrm{Re}(1/k'(z))}\quad\textrm{ and }\quad\left\vert\frac{1}{H'(\frac{y_0}{\sqrt{s}},1)(z)}\right\vert < \frac{c^{1/3}}{\left\vert k'(z)\right\vert}\]
		for all large enough $s$. Therefore, we have
	\begin{align*}
	\left\vert\frac{1}{\mathrm{Re}(1/H_{\frac{y_0}{\sqrt{s}},1}'(z))}-\frac{1}{\mathrm{Re}(1/k'(z))}\right\vert &= \frac{\left\vert\mathrm{Re}(1/k'(z))-\mathrm{Re}(1/H_{\frac{y_0}{\sqrt{s}},1}'(z))\right\vert}{\left\vert\mathrm{Re}(1/H_{\frac{y_0}{\sqrt{s}},1}'(z))\mathrm{Re}(1/k'(z))\right\vert}\\
	&\leq \frac{c^{1/3}}{\left\vert \mathrm{Re}(1/k'(z))\right\vert^2}\frac{c^{1/3}}{\left\vert k'(z)\right\vert^2}\left\vert H'(z) - k'(z)\right\vert\\
	&<\frac{3c\tau(y_0^2)}{s\left\vert z\right\vert^4}\frac{1}{[\mathrm{Re}(1/k'(z))]^2}\frac{1}{\left\vert k'(z)\right\vert^2},
	\end{align*}
	which is the desired inequality since we assume $\tau(y_0^2) = 1$ until the proof of Theorem~\ref{thm:CircularDensity}.
	\end{proof}

\begin{lemma}
	\label{lem:EstFrom2}
	Given any $c>1$ and $\varphi_0\in(0,\pi/2)$, for all sufficient large $s$, the unique $\alpha$ such that
	\[H_{\frac{y_0}{\sqrt{s}},1}(\alpha+iv_{\frac{y_0}{\sqrt{s}},1}(\alpha)) = 2\cos\varphi,\quad \sin\varphi>\sin\varphi_0\]
	satisfies
	\[\left\vert\frac{1}{\mathrm{Re}(1/H_{\frac{y_0}{\sqrt{s}},1}'(w))} - 2\right\vert <\frac{c}{s\sin^2\varphi_0}\left(3+\frac{2}{\sin\varphi_0}\right)\]
	where $w = \alpha+iv_{\frac{y_0}{\sqrt{s}},1}(\alpha)$.
	\end{lemma}
\begin{proof}
	Let $c>1$. Write $z = e^{i\varphi}$ and $w = \alpha+iv_{\frac{y_0}{\sqrt{s}},1}(\alpha)$. Recall that $\frac{1}{\mathrm{Re}(1/k'(z))} = 2$ by~\eqref{eq:k'unit}. We estimate
	\begin{equation}
	\label{eq:EstFrom2}
	\left\vert\frac{1}{\mathrm{Re}(1/H_{\frac{y_0}{\sqrt{s}},1}'(w))} - 2\right\vert\leq \left\vert\frac{1}{\mathrm{Re}(1/H_{\frac{y_0}{\sqrt{s}},1}'(w))} - \frac{1}{\mathrm{Re}(1/k'(w))}\right\vert + \left\vert\frac{1}{\mathrm{Re}(1/k'(w))} - \frac{1}{\mathrm{Re}(1/k'(z))}\right\vert.
	\end{equation}

	We estimate the first term in~\eqref{eq:EstFrom2} using Proposition~\ref{prop:vUpper} and Lemmas~\ref{lem:k'Est} and \ref{lem:H'k'Diff}. Fix any $1<c'<c$. For all large enough~$s$, the first term is bounded by
	\begin{align*}
	\frac{3c'}{s\left\vert w\right\vert^4}\frac{1}{[\mathrm{Re}(1/k'(w))]^2}\frac{1}{\left\vert k'(w)\right\vert^2}&\leq \frac{3c'}{s[1-1/(\sin\varphi_0 s)^4]}\left(\frac{1}{\mathrm{Re}(1/k'(e^{i\varphi}))}+\frac{2c'}{s\sin^3\varphi_0}\right)^2\frac{1}{\left\vert k'(w)\right\vert^2}\\
	&<\frac{12c}{s}\frac{1}{4\sin^2\varphi}\\
	&\leq \frac{3c}{s\sin^2\varphi_0}
	\end{align*}
	by~\eqref{eq:k'=sin} and Lemmas~\ref{lem:k'Est} and~\ref{lem:H'k'Diff}.
	
	By Lemma~\ref{lem:k'Est}, the second term in~\eqref{eq:EstFrom2} is bounded by
	\[\left\vert\frac{1}{\mathrm{Re}(1/k'(w))}-\frac{1}{\mathrm{Re}(1/k'(z))}\right\vert<\frac{2c}{s\sin^3\varphi_0}.\]
	The result then follows from adding these estimates.
	\end{proof}

\begin{proposition}
	\label{prop:c1Density}
	Denote by $w_{\frac{y_0}{\sqrt{s}}, 1}$ the density of $\mathrm{Brown}(\frac{y_0}{\sqrt{s}}+c_1)$. Then, for any $c>1$ and $\varphi_0\in(0,\pi/2)$, we have
	\[\left\vert w_{\frac{y_0}{\sqrt{s}}, 1}(\alpha+i\beta) - \frac{1}{\pi}\right\vert < \frac{c}{2\pi s\sin^2\varphi_0}\left(3+\frac{2}{\sin\varphi_0}\right),\quad \left\vert\psi_{\frac{y_0}{\sqrt{s}},1}(\alpha)\right\vert < 2\cos\varphi_0\]
	for all large enough $s$.
	\end{proposition}
\begin{proof}
	By Equation (3.31) of \cite{HoZhong2019},
	\[\mathrm{Re}\left(\frac{1}{H_{\frac{y_0}{\sqrt{s}},1}'(w)}\right)\frac{d\psi_{\frac{y_0}{\sqrt{s}},1}(\alpha)}{d\alpha} = 1\]
	where $w = \alpha+iv_{\frac{y_0}{\sqrt{s}}, 1}(\alpha)$. (This formula appeals to the subordination function $H_{\frac{y_0}{\sqrt{s}},1}^{-1}$ of the free convolution $\frac{y_0}{\sqrt{s}}+\sigma_1$ has an analytic continuation in a neighborhood of any point $\psi_{\frac{y_0}{\sqrt{s}}, 1}(\alpha+iv_{\frac{y_0}{\sqrt{s}}, 1}(\alpha))$ if $v_{\frac{y_0}{\sqrt{s}}, 1}(\alpha)>0$; see \cite[Theorem 3.3(1)]{Belinschi2008}.) Thus, we can express the real derivative through complex derivative
	\[\frac{d\psi_{\frac{y_0}{\sqrt{s}},1}(\alpha)}{d\alpha} = \frac{1}{\mathrm{Re}(1/H_{\frac{y_0}{\sqrt{s}},1}'(w))}.\]
	
	By Lemma~\ref{lem:EstFrom2}, given any $c>1$ and $\varphi_0\in(0,\pi/2)$, for all sufficient large $s$, the unique $\alpha$ such that
	\[\psi_{\frac{y_0}{\sqrt{s}},1}(\alpha) = 2\cos\varphi,\quad \sin\varphi>\sin\varphi_0\]
	satisfies
	\[\left\vert\frac{d\psi_{\frac{y_0}{\sqrt{s}},1}(\alpha)}{d\alpha} - 2\right\vert <\frac{c}{s\sin^2\varphi_0}\left(3+\frac{2}{\sin\varphi_0}\right).\]
	The proposition now follows from Theorem~\ref{thm:HZ}
	\end{proof}

All the estimates in this section that we have done are under the assumption $\tau(y_0) = 0$ and $\tau(y_0^2)$. We are now ready to prove the estimate of the density of $\mathrm{Brown}(y_0+c_s)$ for arbitrary $\tau(y_0)$ and $\tau(y_0^2)$.

\begin{proof}[{\bf Proof of Theorem \ref{thm:CircularDensity}}]
	Without loss of generality, we assume $\tau(y_0) = 0$, since otherwise we translate the density by~$\tau(y_0)$.
	
	We first assume $\tau(y_0^2) = 1$. Let $w = \alpha + iv_{y_0,s}(\alpha)$ and $z = \frac{w}{\sqrt{s}}$. Then
	\[z = \frac{\alpha}{\sqrt{s}} + iv_{\frac{y_0}{\sqrt{s}},1}\left(\frac{\alpha}{\sqrt{s}}\right).\]
	Since $\mathrm{Brown}(y_0+c_s)$ is the push-forward measure of $\mathrm{Brown}\left(\frac{y_0}{\sqrt{s}}+c_1\right)$ by $z\mapsto \sqrt{s}z$,
	\[w_{y_0,s}(\alpha+i\beta) = \frac{1}{s}\cdot w_{\frac{y_0}{\sqrt{s}},1}\left(\frac{1}{\sqrt{s}}(\alpha+i\beta)\right), \quad z\in\Lambda_{y_0,s}.\] 
	By Proposition~\ref{prop:c1Density}, for any $c>1$ and $\varphi_0\in(0,\pi/2)$, we have
	\[\left\vert w_{y_0,s}(\alpha+i\beta) - \frac{1}{\pi s}\right\vert < \frac{c}{2\pi s^2\sin^2\varphi_0}\left(3+\frac{2}{\sin\varphi_0}\right),\quad \left\vert\psi_{y_0,s}(\alpha)\right\vert < 2\sqrt{s}\cos\varphi_0\]
	for all large enough $s$. This establishes the result with $\tau(y_0^2) = 1$.
	
	For arbitrary $\tau(y_0^2)$, let $Y = \frac{y_0}{\sqrt{\tau(y_0^2)}}$. We consider the random variable $\frac{1}{\sqrt{\tau(y_0^2)}}(y_0+c_s)$
	which has the same $\ast$-moments, hence the same Brown measure, as $Y+c_{t}$, where $t= s/\tau(y_0^2)$.
	
	By the result for $\tau(y_0^2)=1$, given any $c>1$ and $\varphi_0\in(0,\pi/2)$, we have
	\begin{equation}
	\label{eq:CDensityEst}
	\left\vert w_{Y,t}(\alpha+i\beta) - \frac{1}{\pi t}\right\vert <  \frac{c}{2\pi t^2\sin^2\varphi_0}\left(3+\frac{2}{\sin\varphi_0}\right),\quad \left\vert\psi_{Y,t}(\alpha)\right\vert < 2\sqrt{t}\cos\varphi_0
	\end{equation}
	for all large enough $t$. Now, since $\mathrm{Brown}(y_0+c_s)$ is the push-forward measure of $\mathrm{Brown}(Y+c_t)$ by $z\mapsto \sqrt{\tau(y_0^2)}z$, by~\eqref{eq:CDensityEst}, we must have
	\[\left\vert w_{y_0,s}(\alpha+i\beta) - \frac{1}{\pi s}\right\vert <  \frac{c\tau(y_0^2)}{2\pi s^2\sin^2\varphi_0}\left(3+\frac{2}{\sin\varphi_0}\right),\quad \left\vert\psi_{y_0,s}(\alpha)\right\vert < 2\sqrt{s}\cos\varphi_0\]
	for all large enough $s$.
	\end{proof}

\section{Asymptotic behaviors of adding an elliptic element\label{sect:ellipticAsymp}}
In this section, we study three limiting behaviors of $\mathrm{Brown}(y_0+\tilde{\sigma}_{s-t/2}+i\sigma_{t/2})$ as $s\to\infty$. The first regime is to keep $s$ and $t$ at the same ratio $r = t/s$; the second regime is to keep $t$ fixed; the last regime is to fix $s=t/2$. 

\subsection{Fix $s/t$ and let $s, t\to\infty$}

\subsubsection{Domain behavior}
In this section, we discuss the asymptotic behavior of the domain of $\mathrm{Brown}(y_0+\tilde{\sigma}_{s-\frac{t}{2}}+i\sigma_{\frac{t}{2}})$ for a fixed $r = t/s$.
When $y_0=0$, the domain of $\tilde{\sigma}_{s-\frac{t}{2}}+i\sigma_{\frac{t}{2}}$ has the shape of an ellipse with boundary
\begin{equation}
\label{eq:ellipse}
\frac{2s-t}{\sqrt{s}}\cos\varphi+i\frac{t}{\sqrt{s}}\sin\varphi,\quad \varphi\in[0,2\pi]
\end{equation}
(See \cite[Example 5.3]{BianeLehner2001}). As $s\to\infty$ with $r=t/s$ fixed, the random variable $y_0+\tilde{\sigma}_{s-\frac{t}{2}}+i\sigma_{\frac{t}{2}}$ behaves like the elliptic element $\tau(y_0)+\tilde{\sigma}_{s-\frac{t}{2}}+i\sigma_{\frac{t}{2}}$. Roughly speaking, the domain $\Omega_{s,t}$ of $\mathrm{Brown}(y_0+\tilde{\sigma}_{s-\frac{t}{2}}+i\sigma_{\frac{t}{2}})$ is asymptotically an ellipse with boundary as in~\eqref{eq:ellipse} translated by $\tau(y_0)$. The following theorem states precisely the asymptotic behavior of the domain $\Omega_{s,t}$ of $\mathrm{Brown}(y_0+\tilde{\sigma}_{s-\frac{t}{2}}+i\sigma_{\frac{t}{2}})$; the main tool is Theorem~\ref{thm:vsAsymp}. 

\begin{theorem}
	\label{thm:EllipseDomain}
	Fix the ratio $r = t/s$. The following asymptotic behaviors of the graph of $\Omega_{s,t}$ hold.
	\begin{enumerate}
		\item Let $D_\nu = \sup\{\left\vert x-y\right\vert| x, y\in\mathrm{supp}\,\mu\}$. When $s\geq 4 D_\nu^2$, the function $b_{s,t}$ is unimodal.  In particular, $\Omega_{s,t}\cap\R$ is an interval.
		\item 		Given any $c>1$, we have
			\[\left\vert 
		\sup \Omega_{s,t}\cap\R-\left(\tau(y_0)+\frac{2s-t}{\sqrt{s}}\right)\right\vert <\frac{c(3r+2\left\vert 1-r\right\vert)\tau(y_0^2)}{2\sqrt{s}}\]
		and
			\[\left\vert
		\inf \Omega_{s,t}\cap\R-\left(\tau(y_0)-\frac{2s-t}{\sqrt{s}}\right)\right\vert <\frac{c(3r+2\left\vert 1-r\right\vert)\tau(y_0^2)}{2\sqrt{s}}\]
		for all sufficiently large $s$. In particular, $\Lambda_{y_0,s}\cap\R$ is contained in 
		\[\left(\tau(y_0)-\frac{2s-t}{\sqrt{s}}-\frac{c(3r+2\left\vert 1-r\right\vert)\tau(y_0^2)}{2\sqrt{s}}, \tau(y_0)+\frac{2s-t}{\sqrt{s}}+\frac{c(3r+2\left\vert 1-r\right\vert)\tau(y_0^2)}{2\sqrt{s}}\right)\]
		for all large enough $s$.
		\item Given any $\varphi_0\in(0,\pi/2)$, then for all large enough $s$, for all $\left\vert \cos\varphi\right\vert\leq \cos\varphi_0$, the unique $\alpha\in\R$ such that
		\[H_{y_0,s}(\alpha+iv_{y_0,s}(\alpha)) = 2\sqrt{s}\cos\varphi.\]
		satisfies
		\[\left\vert U_{s,t}(\alpha+iv_{y_0,s}(\alpha))-\left[\frac{2s-t}{\sqrt{s}}\cos\varphi+i\frac{t}{\sqrt{s}}\sin\varphi\right]\right\vert<\frac{r}{(\sin\varphi_0)\sqrt{s}}.\]	
	\end{enumerate}	
	\end{theorem}
\begin{proof}
	Point 1 follows directly from \cite[Theorem 3.2]{HasebeUeda2018} which states that $v_{y_0,s}$ is unimodal for $s\geq 4 D_\nu^2$, because, by Proposition~\ref{prop:asthomeo}, we have
	\[b_{s,t} = \frac{t}{s}v_{y_0,s}.\]
	
	Fix $r=t/s$ throughout this proof. 	We now prove Point 2. Without loss of generality, we assume $\tau(y_0) = 0$. We first estimate $a_{1,r}(\alpha^*)$ where
	\[\alpha^* = \sup\Lambda_{y_0/\sqrt{s},1}\cap\R.\]
	We compute
	\begin{equation}
	\label{eq:a1r}
	a_{1,r}(\alpha^*)-(2-r) = (\alpha^*-1)\left(1-\frac{1-r}{\alpha^*}\right)+\frac{(1-r)\tau(y_0^2)}{s(\alpha^*)^3}+\frac{(1-r)}{s^{3/2}}\sum_{n=3}^\infty\frac{\tau(y_0^n)}{s^{(n-3)/2}(\alpha^*)^{n+1}}.
	\end{equation}
	By Proposition~\ref{prop:vEndPoint} (with $s$ replaced by $s/\tau(y_0^2)$), given any $c>1$, for all large enough $s$, we have
	\[\left\vert a_{1,r}(\alpha^*)-(2-r)\right\vert <\frac{c(3r+2\left\vert 1-r\right\vert)\tau(y_0^2)}{2s}.\]
	Since
	\[\sup\Omega_{s, t}\cap\R = \sqrt{s}a_{1,r}(\alpha^*),\]
	we have
	\[\left\vert 
	\sup \Omega_{s,t}\cap\R-\left(\tau(y_0)+\frac{2s-t}{\sqrt{s}}\right)\right\vert <\frac{c(3r+2\left\vert 1-r\right\vert)\tau(y_0^2)}{2\sqrt{s}}\]
	for all sufficiently large $s$. The estimate for $\inf\Omega_{s,t}\cap\R$ is similar.
	
	We prove Point 3 now. By Theorem \ref{thm:pushforward}, we know that
	\[\Omega_{s,t} = U_{s,t}(\Lambda_{y_0,s}).\]
	Suppose $\alpha$ is chosen such that $\psi_{y_0,s}(\alpha) = 2\sqrt{s}\cos\varphi$.	We compute the upper boundary curve $a+ib_{s,t}(a) = U_{s,t}(\alpha+iv_{y_0,s}(\alpha))$ as
	\begin{align*}
	&a_{s,t}(\alpha) = (1-r)\psi_{y_0,s}(\alpha)+r\alpha = 2(1-r)\sqrt{s}\cos\varphi+r\alpha;\\
	&b_{s,t}(a) = b_{s,t}(a_{s,t}(\alpha)) = r v_{y_0,s}(\alpha).
	\end{align*}
	So, we have
	\begin{equation}
	\label{eq:EllipseImag}
	\left\vert a+ib_{s,t}(a)-\sqrt{s}[(2-r)\cos\varphi+ir\sin\varphi]\right\vert = r\left\vert\alpha+iv_{y_0,s}(\alpha)-\sqrt{s}e^{i\varphi}\right\vert.
	\end{equation}
	Therefore, by Theorem~\ref{thm:vsAsymp}, for any $\varphi_0\in(0,\pi/2)$,
	\begin{align*}
	\left\vert a+ib_{s,t}(a)-\sqrt{s}[(2-r)\cos\varphi+ir\sin\varphi]\right\vert &= r\left\vert\alpha+iv_{y_0,s}(\alpha)-\sqrt{s}e^{i\varphi}\right\vert\\
	&< \frac{r}{(\sin\varphi_0)\sqrt{s}}
	\end{align*}
	for all sufficiently large $s$. This proves Point 3.

	\end{proof}

\subsubsection{Density behavior}
In this section, we investigate the asymptotic behavior of the density of $\mathrm{Brown}(y_0+\tilde{\sigma}_{s-\frac{t}{2}}+i\sigma_{\frac{t}{2}})$ for a fixed  $r = t/s$. In the case $y_0 = 0$, $\mathrm{Brown}(y_0+\tilde{\sigma}_{s-\frac{t}{2}}+i\sigma_{\frac{t}{2}})$ is the elliptic law, with constant density
\begin{equation}
\label{eq:ellipticLaw}
\frac{1}{\pi}\frac{s}{(2s-t)t}
\end{equation}
in domain $\Omega_{s,t}$, which is a region bounded by an ellipse in this case (See \cite[Example 5.3]{BianeLehner2001}). 

Denote by $w_{y_0,s,t}$ the density of $\mathrm{Brown}(y_0+\tilde{\sigma}_{s-\frac{t}{2}}+i\sigma_{\frac{t}{2}})$. We will prove that as $s$ large and $r=t/s$ fixed, the density $w_{y_0,s,t}$ is approximately the same constant in~\eqref{eq:ellipticLaw}.  The main tool is the estimate of the density of $\mathrm{Brown}(y_0+c_s)$ in Theorem~\ref{thm:CircularDensity}.
\begin{theorem}
	\label{thm:EllipticDensity}
	Fix $r = t/s$. Given any $c>1$ and $\varphi_0\in(0,\pi/2)$, we have
	\[\left\vert w_{y_0,s,t}(a+ib) - \frac{1}{\pi}\frac{s}{(2s-t)t}\right\vert<\frac{c\tau(y_0^2)}{2\pi \sin^2\varphi_0}\frac{1}{(2s-t)^2}\left(3+\frac{2}{\sin\varphi_0}\right)\]
	whenever $\psi_{y_0,s}(\alpha_{s,t}(a))<2\sqrt{s}\cos\varphi_0$, for all large enough $s$.
	\end{theorem}
\begin{proof}
Let $c>1$ be given. By Corollary~\ref{cor:stFormulas}, if we write $a+ib = U_{s,t}(\alpha+i \beta)$ for all $\alpha+i\beta\in\Lambda_{y_0,s}$. Then we have
	\[w_{y_0,s,t}(a+ib) = \frac{1}{r}\frac{w_{y_0,s}(\alpha+i\beta)}{r+2\pi(1-r)s\cdot w_{y_0,s}(\alpha+i\beta)}\]
	for all $a+ib\in\Omega_{s,t}$.
	
	Now, by the formula
	\[ \frac{1}{\pi s}\frac{1}{2-r}=\frac{1/(\pi s)}{r+2\pi(1-r)s\cdot (1/\pi s)},\]
	and Theorem~\ref{thm:CircularDensity}, for any $1<c'<c$, if $\psi_{y_0,s}(\alpha)<2\sqrt{s}\cos\varphi_0$, then we have $\pi s w_{y_0,s}(\alpha+i\beta)\to 1$, and
	\begin{align*}
	\left\vert\frac{w_{y_0,s}(\alpha+i\beta)}{r+2\pi(1-r)s\cdot w_{y_0,s}(\alpha+i\beta)} -\frac{1/(\pi s)}{2-r}\right\vert
	&= \frac{r\left\vert w_{y_0,s}(\alpha+i\beta)-1/(\pi s)\right\vert}{[r+2\pi(1-r)s\cdot w_s(\alpha+i\beta)][2-r]}\\
	&<\frac{cr\tau(y_0^2)}{2\pi s^2\sin^2\varphi_0}\left(3+\frac{2}{\sin\varphi_0}\right)\frac{1}{(2-r)^2}
	\end{align*}
	for all large enough $s$. The proof follows from dividing the above estimate by $r$.
	\end{proof}

\subsection{Fix $t$ and let $s\to\infty$}
In this section, we investigate the asymptotic behavior of $\mathrm{Brown}(y_0+\tilde{\sigma}_{s-t/2}+i\sigma_{t/2})$ with $t$ fixed and $s\to\infty$. 
\subsubsection{Domain behavior}
The following theorem states that $\Omega_{s,t}$ has the shape of an ellipse in the limit with fixed $t$ as $s\to\infty$, except points close to the \emph{endpoints} of $\Omega_{s,t}\cap\R$. The limiting ellipse has a very short minor axis; it is a long and thin ellipse.
\begin{theorem}
	\label{thm:thm:FixtAsympDomain}
	Fix $t>0$. The following asymptotic behaviors of the graph of $\Omega_{s,t}$ hold.
	\begin{enumerate}
		\item Let $D_\nu = \sup\{\left\vert x-y\right\vert| x, y\in\mathrm{supp}\,\mu\}$. When $s\geq 4 D_\nu^2$, the function $b_{s,t}$ is unimodal.  In particular, $\Omega_{s,t}\cap\R$ is an interval.
		\item 		Given any $c>1$, we have
			\[\left\vert 
			\sup \Omega_{s,t}\cap\R-\left(\tau(y_0)+2\sqrt{s}\right)\right\vert <\frac{c\left\vert \tau(y_0^2)-t\right\vert}{\sqrt{s}}\]
		and
		\[\left\vert
		\inf \Omega_{s,t}\cap\R-\left(\tau(y_0)-2\sqrt{s}\right)\right\vert <\frac{c\left\vert \tau(y_0^2)-t\right\vert}{\sqrt{s}}\]
		for all sufficiently large $s$. In particular,
		\[ \Lambda_{y_0,s}\cap\R\subset \left(\tau(y_0)-2\sqrt{s}-\frac{c\left\vert \tau(y_0^2)-t\right\vert}{\sqrt{s}}, \tau(y_0)+2\sqrt{s}+\frac{c\left\vert \tau(y_0^2)-t\right\vert}{\sqrt{s}}\right)\]
		for all large enough $s$.
		\item Given any $\varphi_0\in(0,\pi/2)$, then for all large enough $s$, for all $\left\vert \cos\varphi\right\vert\leq \cos\varphi_0$, the unique $\alpha\in\R$ such that
		\[H_{y_0,s}(\alpha+iv_{y_0,s}(\alpha)) = 2\sqrt{s}\cos\varphi.\]
		satisfies
		\[\left\vert U_{s,t}(\alpha+iv_{y_0,s}(\alpha))-\left[\frac{2s-t}{\sqrt{s}}\cos\varphi+i\frac{t}{\sqrt{s}}\sin\varphi\right]\right\vert<\frac{t}{(\sin\varphi_0)s^{3/2}}.\]	
		Furthermore, we have
		\[\lim_{s\to \infty}\sup\{\left\vert \mathrm{Im}\,z\right\vert | z\in\Omega_{s,t}\} = 0.\]
	\end{enumerate}	
\end{theorem}
\begin{proof}
	Point 1 follows directly from Theorem~\ref{thm:pushforward} and \cite[Theorem 3.2]{HasebeUeda2018} which states that $v_{y_0,s}$ is unimodal for $s\geq 4 D_\nu^2$, because, by \eqref{eq:ParaBoundary}, we have
	\[b_{s,t} = \frac{t}{s}v_{y_0,s}.\] 
	
	Fix $t>0$. We now prove Point 2. Without loss of generality, we assume $\tau(y_0) = 0$. We first estimate $a_{1,r}(\alpha^*)$ where
	\[\alpha^* = \sup\Lambda_{y_0/\sqrt{s},1}\cap\R.\]
	We calculate
	\begin{align*}
	a_{1,r}(\alpha^*)-2 &= \alpha^*-2+(1-r)\sum_{n=0}^\infty\frac{\tau(y_0^n)}{s^{\frac{n}{2}}(\alpha^*)^{n+1}}\\
	&=\alpha^*-1+\frac{1-\alpha^*}{\alpha^*}-\frac{t}{s\alpha^*}+\frac{\tau(y_0^2)}{s(\alpha^*)^3}+\sum_{n=3}^\infty\frac{\tau(y_0^n)}{s^{\frac{n}{2}}(\alpha^*)^{n+1}}\\
	&=(\alpha^*-1)\frac{\alpha^*-1}{\alpha^*}+\frac{\tau(y_0^2)-t(\alpha^*)^2}{s(\alpha^*)^3}+\sum_{n=3}^\infty\frac{\tau(y_0^n)}{s^{\frac{n}{2}}(\alpha^*)^{n+1}}
          	\end{align*}
	By Proposition~\ref{prop:vEndPoint} (with $s$ replaced by $s/\tau(y_0^2)$), given any $c>1$, for all large enough $s$, we have (by keeping the only order $1/s$ term)
	\[\left\vert a_{1,r}(\alpha^*)-2\right\vert <\frac{c\left\vert \tau(y_0^2)-t\right\vert}{s}.\]
	It follows that
	\[\left\vert 
	\sup \Omega_{s,t}\cap\R-\left(\tau(y_0)+2\sqrt{s}\right)\right\vert <\frac{c\left\vert \tau(y_0^2)-t\right\vert}{\sqrt{s}}\]
	for all sufficiently large $s$. The estimate for $\inf\Omega_{s,t}\cap\R$ is similar.
	
	We now prove Point 3. By~\eqref{eq:EllipseImag},
	\[\left\vert a+ib_{s,t}(a)-\sqrt{s}[(2-r)\cos\varphi+ir\sin\varphi]\right\vert = r\left\vert\alpha+iv_{y_0,s}(\alpha)-\sqrt{s}e^{i\varphi}\right\vert.
	\]
	Therefore, by Theorem~\ref{thm:vsAsymp}, for any $\varphi_0\in(0,\pi/2)$,
	\begin{equation}
	\label{eq:FixtEst}
	\begin{split}
	\left\vert a+ib_{s,t}(a)-\sqrt{s}[(2-r)\cos\varphi+ir\sin\varphi]\right\vert< \frac{t}{(\sin\varphi_0)s^{3/2}}
	\end{split}
	\end{equation}
	for all sufficiently large $s$. 
	
	Let $\varphi_0=\frac{\pi}{6}$ so that $\sin\varphi>1/2$ for all $\varphi$ such that $\left\vert\cos\varphi\right\vert<\cos\varphi_0$. We label by $\alpha_\varphi$ the unique $\alpha\in\R$ such that
	\[H_{y_0,s}(\alpha+iv_{y_0,s}(\alpha)) = 2\sqrt{s}\cos\varphi, \quad\left\vert\cos\varphi\right\vert\leq \cos\varphi_0.\]
	By~\eqref{eq:FixtEst}, we have
	\begin{equation*}
	\sup \{b_{s,t}(a_{s,t}(\alpha))\vert \alpha_{\pi-\varphi_0}<\alpha<\alpha_{\varphi_0}\}>\frac{t}{\sqrt{s}}-\frac{2t}{s^{3/2}}.
	\end{equation*}
	Since
	\[b_{s,t}(a_{s,t}(\alpha_{\varphi_0}))<\frac{t}{2\sqrt{s}}+\frac{2t}{s^{3/2}}\]
	and, by Point 1, the function $b_{s,t}$ is unimodal,
	\begin{equation}
	\label{eq:EstOutside}
	b_{s,t}(a_{s,t}(\alpha))<\frac{t}{2\sqrt{s}}+\frac{2t}{s^{3/2}},\quad \alpha\geq\alpha_{\varphi_0}\textrm{ or }\alpha\leq\alpha_{\pi-\varphi_0}.
	\end{equation}
	For all $\alpha_{\pi-\varphi_0}<\alpha<\alpha_{\varphi_0}$,
	\begin{equation}
	\label{eq:supabove}
	\sup \{b_{s,t}(a_{s,t}(\alpha))\vert \alpha_{\pi-\varphi_0}<\alpha<\alpha_{\varphi_0}\}<\frac{t}{\sqrt{s}}+\frac{2t}{s^{3/2}}.
	\end{equation}
	Therefore, we conclude
	\[\lim_{s\to \infty}\sup\{\left\vert \mathrm{Im}\,z\right\vert | z\in\Omega_{s,t}\} = 0\]
	by~\eqref{eq:EstOutside} and~\eqref{eq:supabove}.
	\end{proof}
\subsubsection{Density behavior}
If we consider the special case of $y_0 = 0$, $\mathrm{Brown}(y_0+\tilde\sigma_{s-t/2}+i\sigma_{t/2})$ is just the elliptic law; as mentioned in ~\eqref{eq:ellipticLaw}, it has a constant density
\[\frac{1}{\pi}\frac{s}{(2s-t)t}.\]
If we fixed $t$ and let $s\to\infty$, this density converges to the constant $1/(2\pi t)$. 

The following theorem states that if we consider an arbitrary self-adjoint initial condition $y_0$, the density of $\mathrm{Brown}(y_0+\tilde\sigma_{s-t/2}+i\sigma_{t/2})$ also converges to $1/(2\pi t)$; the convergence is uniform away the endpoints of $\Omega_{s,t}\cap\R$.
\begin{theorem}
	\label{thm:FixtAsympDensity}
	Denote by $w_{y_0,s,t}$ the density of $\mathrm{Brown}(y_0+\tilde\sigma_{s-t/2}+i\sigma_{t/2})$. Then given any $c>1$ and $\varphi_0\in(0,\pi/2)$, there is an $s_0>0$ such that 
	\[\left\vert w_{y_0,s,t}(a+ib) - \frac{1}{2\pi t}\right\vert <\frac{c}{4\pi s}, \quad\left\vert\psi_{y_0,s}(\alpha_{s,t}(a))\right\vert < 2\sqrt{s}\cos\varphi_0\]
	for all $s>s_0$.
	\end{theorem}
\begin{proof}
	Let $c>1$ and $\varphi_0\in(0,\pi/2)$ be given. By Corollary~\ref{cor:stFormulas}, if we write $(a,b) = U_{s,t}(\alpha, \beta)$ for all $\alpha+i\beta\in\Lambda_{y_0,s}$. Then we have
	\begin{equation}
	\label{eq:Densityst}
	w_{y_0,s,t}(a+ib) = \frac{1}{2\pi t}\frac{s \pi w_{y_0,s}(\alpha+i\beta)}{t/(2s)+(1-t/s)\pi s\cdot w_{y_0,s}(\alpha+i\beta)}
	\end{equation}
	for all $a+ib\in\Omega_{s,t}$.
	
	By Theorem~\ref{thm:CircularDensity}, given any $1<c'<c$, we have
	\[\left\vert \pi s\cdot w_{y_0,s}(\alpha+i\beta) - 1\right\vert < \frac{c'\tau(y_0^2)}{2 s\sin^2\varphi_0}\left(3+\frac{2}{\sin\varphi_0}\right),\quad \left\vert\psi_{y_0,s}(\alpha)\right\vert < 2\sqrt{s}\cos\varphi_0\]
	for all large enough $s$. Then, we compute
	\begin{align*}
	\left\vert\frac{s \pi w_{y_0,s}(\alpha+i\beta)}{t/(2s)+(1-t/s)\pi s\cdot w_{y_0,s}(\alpha+i\beta)} - 1 \right\vert&=\frac{t}{s}\left\vert \frac{\pi s\cdot w_{y_0,s}(\alpha+i\beta)-1/2}{t/(2s)+(1-t/s)\pi s\cdot w_{y_0,s}(\alpha+i\beta)}\right\vert\\
	&< \frac{c't}{s}\left[\frac{1}{2}+\frac{c'\tau(y_0^2)}{2 s\sin^2\varphi_0}\left(3+\frac{2}{\sin\varphi_0}\right)\right]\\
	&<\frac{ct}{2s}.
	\end{align*}
	for all large enough $s$, since $t/(2s)+(1-t/s)\pi s\cdot w_{y_0,s}(\alpha+i\beta)$ converges to $1$. Thus, using~\eqref{eq:Densityst}, we have the estimate (uniform for all $\left\vert\psi_{y_0,s}(\alpha_{s,t}(a))\right\vert < 2\sqrt{s}\cos\varphi_0$)
	\[w_{y_0,s,t}(a+ib) - \frac{1}{2\pi t} = \frac{1}{2\pi t}\left\vert\frac{s \pi w_{y_0,s}(\alpha+i\beta)}{t/(2s)+(1-t/s)\pi s\cdot w_{y_0,s}(\alpha+i\beta)} - 1 \right\vert <\frac{c}{4\pi s}\]
	for all sufficiently large $s$.
\end{proof}

\subsection{Set $s=t/2$ and let $s\to\infty$}
In this section, we investigate the asymptotic behavior of $\mathrm{Brown}(y_0+\sigma_{s-t/2}+i\tilde{\sigma}_{t/2})$ with $s=t/2$ and $s\to\infty$. Note that, when $s=t/2$, the random variable $y_0+\tilde{\sigma}_{s-t/2}+i\sigma_{t/2}$ is $y_0+i\sigma_{s}$. 

\begin{theorem}
	\label{thm:Skewasymptotic}
	\begin{enumerate}
		\item Let $D_\nu = \sup\{\left\vert x-y\right\vert| x, y\in\mathrm{supp}\,\mu\}$. When $s\geq 4 D_\nu^2$, the function $b_{s,t}$ is unimodal. In particular, $\Omega_{s,t}\cap\R$ is an interval.
		\item We have
			\[-\frac{4c\tau(y_0^2)}{\sqrt{s}}<\inf(\Omega_{s,t}\cap\R)-\tau(y_0)<0<\sup(\Omega_{s,t}\cap\R)-\tau(y_0)<\frac{4c\tau(y_0^2)}{\sqrt{s}}\]
		for all $s$ large enough. In particular,
		\[\Omega_{s,t}\cap \R\subset \left( \tau(y_0)- \frac{4c\tau(y_0^2)}{\sqrt{s}},  \tau(y_0)+ \frac{4c\tau(y_0^2)}{\sqrt{s}}\right)\]
		for all $s$ large enough.
		\item We also have
		\[\left\vert\sup\{\left\vert\mathrm{Im}\,z\right\vert|z\in\Omega_{s,t}\} - 2\sqrt{s}\right\vert<\frac{2c}{\sqrt{s}}\]
		for all large enough $s$.
	\end{enumerate}
\end{theorem}
\begin{proof}
	Point 1 follows directly from \cite[Theorem 3.2]{HasebeUeda2018} which states that $v_{y_0,s}$ is unimodal for $s\geq 4 D_\nu^2$, because, \eqref{eq:ParaBoundary}, we have
	\[b_{s,t} = 2v_{y_0,s}.\]
	
	We now prove Point 2. Let $c>1$ be given. Without loss of generality, we assume $\tau(y_0) = 0$. Denote
	\[M_s=\sup(\Lambda_s\cap\mathbb{R})\quad \textrm{and}\quad m_s=\inf(\Lambda_s\cap\mathbb{R}).\]
	Then $\sup(\Omega_{y_0,s}\cap\R) = a_{y_0,s}(M_s)$ and $\inf(\Omega_{y_0,s}\cap\R) = a_{y_0,s}(m_s)$. First, $M_s > \sup(\mathrm{supp}\,\nu)$ by Point 1 of Theorem \ref{thm:vsAsymp}. Recall from Definition \ref{def:H} that (since $M_t$ is real)
	\begin{equation}
	\label{eq:asEst}
	\begin{split}
	a_{y_0, s}(M_s) &= H_{y_0,-s}(M_s)\\
	 &= M_s - s\int\frac{d\nu(x)}{M_s-x}\\
	&= \frac{1}{M_s}(M_s^2-s)- \frac{s}{M_s^3}\sum_{n=2}^\infty\frac{\tau(y_0^n)}{M_s^{n-2}}.
	\end{split}
	\end{equation}
	Now, by Theorem~\ref{thm:vsAsymp}, we have
	\[\sqrt{s} - \frac{3c\tau(y_0^2)}{2\sqrt{s}}<M_s < \sqrt{s} + \frac{3c'\tau(y_0^2)}{2\sqrt{s}}\]
	for all large enough $s$. Thus we can estimate $\left\vert a_{y_0, s}(M_t)\right\vert $ by \eqref{eq:asEst}
	\begin{align*}
	\left\vert a_{y_0,s}(M_s)\right\vert & = \left\vert(M_s-\sqrt{s})\left(1+\frac{\sqrt{s}}{M_s}\right)- \frac{s}{M_s^3}\sum_{n=2}^\infty\frac{\tau(y_0^n)}{M_s^{n-2}}\right\vert\\
	&< \frac{3c\tau(y_0^2)}{\sqrt{s}}+\frac{c\tau(y_0^2)}{\sqrt{s}} \\
	&= \frac{4c\tau(y_0^2)}{\sqrt{s}}.
	\end{align*}
	By that $\mathrm{Brown}(y_0+i\sigma_s)$ is symmetric about the real axis and the \emph{holomorphic} moments of $\mathrm{Brown}(y_0+i\sigma_s)$ agree with the corresponding holomorphic moments of $y_0+i\sigma_s$ \cite{Brown1986},
	\begin{equation}
	\label{eq:BrownInt}
	\begin{split}
	\int a \,d\mathrm{Brown}(y_0+i\sigma_s)(a+ib) &=\int (a+ib) \,d\mathrm{Brown}(y_0+i\sigma_s)(a+ib) \\
	&= \tau(y_0+i\sigma_s) = 0.
	\end{split}
	\end{equation}
	It is impossible that $a_{y_0,s}(M_s) \leq 0$; otherwise, since $\Omega_{y_0,s}$ is not a subset of the imaginary axis, the integral in \eqref{eq:BrownInt} is negative, contradicting that the integral is $0$.
	
	The estimate for $a_{y_0,s}(m_s)$ is similar.
	
	To prove Point 3, we let $\varphi_0\in(0,\pi/2)$ such that $1/(\sin\varphi_0)<c$. By Theorem~\ref{thm:vsAsymp}, if we write $
	\alpha_\varphi$ the unique real number such that
	\[H_{y_0,s}(\alpha_\varphi+iv_{y_0,s}(\alpha_\varphi)) = 2\sqrt{s}\cos\varphi,\quad\left\vert\cos\varphi\right\vert\leq\cos\varphi_0,\]
	then
	\[\left\vert\alpha_\varphi+iv_{y_0,s}(\alpha_\varphi) - \sqrt{s}e^{i\varphi}\right\vert<\frac{1}{(\sin\varphi_0)\sqrt{s}}.\]
	Thus, we have
	\[\sqrt{s}-\frac{1}{(\sin\varphi_0)\sqrt{s}}<\sup\{v_{y_0,s}(\alpha_\varphi)|\left\vert\cos\varphi\right\vert<\cos\varphi_0\}<\sqrt{s}+\frac{1}{(\sin\varphi_0)\sqrt{s}}.\]
	Also, for all $\alpha\geq\alpha_{\varphi_0}$ or $\alpha\leq\alpha_{\pi-\varphi_0}$, we have, by unimodality of $v_{y_0,s}$,
	\begin{align*}
	v_{y_0,s}(\alpha)&<\sqrt{s}\sin\varphi_0+\frac{1}{(\sin\varphi_0)\sqrt{s}}\\
	&<\sqrt{s}-\frac{1}{\sqrt{s}\sin\varphi_0}\\
	&<\sup\{v_{y_0,s}(\alpha_\varphi)|\left\vert\cos\varphi\right\vert<\cos\varphi_0\}
	\end{align*}
	for all large enough $s$. It follows that
	\[\left\vert\sup_{\alpha\in\R} v_{y_0, s}(\alpha) -\sqrt{s}\right\vert<\frac{1/(\sin\varphi_0)}{\sqrt{s}}<\frac{c}{\sqrt{s}}\]
	for all sufficiently large $s$. Because $b_{s,t} = 2v_{y_0,s}$,
	Point 3 of this theorem is established.
\end{proof}

\section{Acknowledgments}
The author would like to thank Hari Bercovici, Brian Hall, and Ping Zhong for useful conversations. The author also thanks the anonymous referees for carefully reading the paper thoroughly and for making valuable suggestions and corrections.

\bibliographystyle{acm}
\bibliography{AdditiveElliptic}

\end{document}